\documentclass[12pt]{amsart}
\usepackage{amssymb}
\usepackage{graphics}
\usepackage{latexsym}
\usepackage{amsmath}
\usepackage{amssymb,amsthm,amsfonts}
\usepackage{mathtools}
\usepackage{amscd}
\usepackage[arrow, matrix, curve]{xy}
\usepackage{syntonly}
\title[]{Periodic Higgs subbundles in positive and mixed characteristic}
\author[Mao Sheng]{Mao Sheng}
\email{msheng@ustc.edu.cn} \address{School of Mathematical Sciences,
University of Science and Technology of China, Hefei, 230026, China}
\author[Kang Zuo]{Kang Zuo}
\email{zuok@uni-mainz.de}
\address{Institut f\"{u}r  Mathematik, Universit\"{a}t
Mainz, Mainz, 55099, Germany}

\begin{document}
\theoremstyle{plain}
\newtheorem{thm}{Theorem}[section]
\newtheorem{theorem}[thm]{Theorem}
\newtheorem{lemma}[thm]{Lemma}
\newtheorem{corollary}[thm]{Corollary}
\newtheorem{proposition}[thm]{Proposition}
\newtheorem{addendum}[thm]{Addendum}
\newtheorem{variant}[thm]{Variant}
\theoremstyle{definition}
\newtheorem{lemma and definition}[thm]{Lemma and Definition}
\newtheorem{construction}[thm]{Construction}
\newtheorem{notation}[thm]{Notation}
\newtheorem{question}[thm]{Question}
\newtheorem{problem}[thm]{Problem}
\newtheorem{remark}[thm]{Remark}
\newtheorem{remarks}[thm]{Remarks}
\newtheorem{definition}[thm]{Definition}
\newtheorem{claim}[thm]{Claim}
\newtheorem{assumption}[thm]{Assumption}
\newtheorem{assumptions}[thm]{Assumptions}
\newtheorem{properties}[thm]{Properties}
\newtheorem{example}[thm]{Example}
\newtheorem{conjecture}[thm]{Conjecture}
\newtheorem{proposition and definition}[thm]{Proposition and Definition}
\numberwithin{equation}{thm}

\newcommand{\pP}{{\mathfrak p}}
\newcommand{\sA}{{\mathcal A}}
\newcommand{\sB}{{\mathcal B}}
\newcommand{\sC}{{\mathcal C}}
\newcommand{\sD}{{\mathcal D}}
\newcommand{\sE}{{\mathcal E}}
\newcommand{\sF}{{\mathcal F}}
\newcommand{\sG}{{\mathcal G}}
\newcommand{\sH}{{\mathcal H}}
\newcommand{\sI}{{\mathcal I}}
\newcommand{\sJ}{{\mathcal J}}
\newcommand{\sK}{{\mathcal K}}
\newcommand{\sL}{{\mathcal L}}
\newcommand{\sM}{{\mathcal M}}
\newcommand{\sN}{{\mathcal N}}
\newcommand{\sO}{{\mathcal O}}
\newcommand{\sP}{{\mathcal P}}
\newcommand{\sQ}{{\mathcal Q}}
\newcommand{\sR}{{\mathcal R}}
\newcommand{\sS}{{\mathcal S}}
\newcommand{\sT}{{\mathcal T}}
\newcommand{\sU}{{\mathcal U}}
\newcommand{\sV}{{\mathcal V}}
\newcommand{\sW}{{\mathcal W}}
\newcommand{\sX}{{\mathcal X}}
\newcommand{\sY}{{\mathcal Y}}
\newcommand{\sZ}{{\mathcal Z}}
\newcommand{\A}{{\mathbb A}}
\newcommand{\B}{{\mathbb B}}
\newcommand{\C}{{\mathbb C}}
\newcommand{\D}{{\mathbb D}}
\newcommand{\E}{{\mathbb E}}
\newcommand{\F}{{\mathbb F}}
\newcommand{\G}{{\mathbb G}}
\newcommand{\HH}{{\mathbb H}}
\newcommand{\I}{{\mathbb I}}
\newcommand{\J}{{\mathbb J}}
\renewcommand{\L}{{\mathbb L}}
\newcommand{\M}{{\mathbb M}}
\newcommand{\N}{{\mathbb N}}
\renewcommand{\P}{{\mathbb P}}
\newcommand{\Q}{{\mathbb Q}}
\newcommand{\R}{{\mathbb R}}
\newcommand{\SSS}{{\mathbb S}}
\newcommand{\T}{{\mathbb T}}
\newcommand{\U}{{\mathbb U}}
\newcommand{\V}{{\mathbb V}}
\newcommand{\W}{{\mathbb W}}
\newcommand{\X}{{\mathbb X}}
\newcommand{\Y}{{\mathbb Y}}
\newcommand{\Z}{{\mathbb Z}}
\newcommand{\id}{{\rm id}}
\newcommand{\rank}{{\rm rank}}
\newcommand{\END}{{\mathbb E}{\rm nd}}
\newcommand{\End}{{\rm End}}
\newcommand{\Hom}{{\rm Hom}}
\newcommand{\Hg}{{\rm Hg}}
\newcommand{\tr}{{\rm tr}}
\newcommand{\Cor}{{\rm Cor}}
\newcommand{\GL}{\mathrm{GL}}
\newcommand{\SL}{\mathrm{SL}}
\newcommand{\Aut}{\mathrm{Aut}}
\newcommand{\Sym}{\mathrm{Sym}}
\newcommand{\DD}{\mathbf{D}}
\newcommand{\EE}{\mathbf{E}}
\newcommand{\Gal}{\mathrm{Gal}}
\newcommand{\GSp}{\mathrm{GSp}}
\newcommand{\Spf}{\mathrm{Spf}}
\newcommand{\Spec}{\mathrm{Spec}}
\newcommand{\SU}{\mathrm{SU}}
\newcommand{\Res}{\mathrm{Res}}
\newcommand{\Rep}{\mathrm{Rep}}
\newcommand{\Span}{\mathrm{Span}}
\newcommand{\SOV}{\mathrm{SOV}}
\newcommand{\ord}{\mathrm{ord}}

\thanks{This work is supported by the SFB/TR 45 `Periods, Moduli
Spaces and Arithmetic of Algebraic Varieties' of the DFG, and
partially supported by the University of Science and Technology of
China.}

\begin{abstract}
Let $k$ be an algebraically closed field of odd characteristic $p$
and  $X$ a proper smooth scheme over the Witt ring $W(k)$. To an
object $(M,Fil^{\cdot},\nabla,\Phi)$ in the Faltings category
$\mathcal{MF}^{\nabla}_{[0,n]}(X), n\leq p-2$, one associates an
\'{e}tale local system $\V$ over the generic fiber of $X$ and a
Higgs bundle $(E,\theta)$ over $X$. Our motivation is to find the
analogue of the classical Simpson correspondence for the categories
of subobjects of $\V$ and $(E,\theta)$. Our main discovery in this
paper is the notion of periodic Higgs subbundles, both in positive
characteristic and in mixed characteristic. In char $p$, it relies
on the inverse Cartier transform constructed by Ogus and Vologodsky
in their work on the char $p$ nonabelian Hodge theory. A lifting of
the inverse Cartier transform to mixed characteristic is
constructed, which is used for the notion of periodicity in mixed
characteristic. We show a one to one correspondence between the set
of periodic Higgs subbundles of $(E,\theta)$ and the set of
\'{e}tale sub local systems of $\V\otimes_{\Z_{p}}\Z_{p^r}$, where
$r$ is a natural number. The notion turns out to be useful in
applications. We have proven, among other results, that the
reduction $(E,\theta)_0$ of $(E,\theta)$ modulo $p$ is Higgs stable,
if and only if, the corresponding representation $\V$ is absolutely
irreducible over $k$.
\end{abstract}

\maketitle

\tableofcontents

\section{Introduction}
Let $\V$ be a complex polarizable variation of Hodge structures
(abbreviated as $\C$-PVHS) over a projective algebraic manifold and
$(E,\theta)$ the corresponding Higgs bundle (see \S1 \cite{De}, \S4 \cite{Si}). Then $(E,\theta)$ is
Higgs polystable of slope zero, that is,
$(E,\theta)=\oplus_{i}(G_i,\theta_i)$ with $(G_i,\theta_i)$ Higgs
stable of slope zero. Each direct factor corresponds to a sub
$\C$-PVHS, since the Hodge metric is indeed Hermitian-Yang-Mills by
the curvature formula due to P. Griffiths (see Theorem 5.2
\cite{Gr}). In this paper we intend to work out a char $p$ as well
as $p$-adic analogue of this result. Our guiding principle is that
the relative Frobenius in the $p$-adic case is a replacement of the
Hodge metric in the complex case.
\begin{notation}
For a nonnegative integer $m$, we denote the reduction of an object
defined over $W:=W(k)$ modulo $p^{m+1}$ by attaching the subscript $m$. As
an example, $X_0$ means the mod $p$ reduction of $X$, i.e, the
closed fiber of $X$ over $W$.
\end{notation}
A good $p$-adic analogue of the category of $\C$-PVHSs is the
category $\mathcal{MF}^{\nabla}_{[0,n]}(X)$ with $n\leq p-2$
(abbreviated as $\mathcal{MF}^\nabla$) introduced by G. Faltings
(see \cite{Fa2}-\cite{Fa1} and \S\ref{MF category} for details). Assume $X$ is a smooth projective $W$-scheme with connected geometric generic fiber. Choose and then fix a smooth projective $W$-curve $Z\subset X$ which is a complete intersection of a very ample divisor of $X$ over $W$. The slope for a vector bundle over $X_0$ in this paper means the $\mu_{Z_0}$-slope.
For an object $(M,\nabla,Fil^{\cdot},\Phi)\in \mathcal{MF}^\nabla$,
there are two associated objects: on the one hand, Faltings loc.
cit. associates it to a representation $\V$ of the arithmetic
fundamental group of the generic fiber
$X^0:=X\times_W\mathrm{Frac}(W)$. On the other hand, because of
Griffiths transversality, one associates a Higgs bundle $(E,\theta)$
by taking grading of $(M,\nabla)$ with respect to the filtration
$Fil^{\cdot}$. Explicitly,
$$
E=\oplus_{i=0}^{n} E^{i,n-i}, \theta=\oplus_{i=0}^{n}
\theta^{i,n-i},
$$
where $E^{i,n-i}=Fil^i/Fil^{i+1}$ and the connection $\nabla$
induces an $\sO_X$-morphism
$$
\theta^{i,n-i}: E^{i,n-i}\to E^{i-1,n-i+1}\otimes \Omega_{X|W}.
$$
Borrowing a terminology in complex case (see \S4 \cite{Si}), we call a
Higgs bundle of the above form a system of Hodge bundles. This assumption is however not restrictive. We have
taken the first step by showing the following result.
\begin{proposition}[Proposition 0.2 \cite{SXZ}]
Notation as above. The Higgs bundle $(E,\theta)_0$ over $X_0$ is
Higgs semistable of slope zero.
\end{proposition}
The result was first shown by Ogus and Vologodsky in the curve case
under a stronger assumption on $p$ (see Proposition 4.19 \cite{OV}).
In this paper we obtain a further result.
\begin{theorem}[Theorem \ref{irreducibility implies stability}]
The Higgs bundle $(E,\theta)_0$ is Higgs stable if and only if the
representation $\V\otimes k$ is irreducible.
\end{theorem}
The theorem is further generalized in \S\ref{further applications}.
We obtain some results in mixed characteristic as well. Among other
results, we have proven the following
\begin{theorem}[Theorem \ref{lifting statement of prop in geo case}]\label{main theorem 1}
Suppose that $M\in \mathcal{MF}^\nabla$ arises from geometry and is
non $p$-torsion. Suppose furthermore that $X_{0}(k)$ contains an
ordinary point. If there is a decomposition of Higgs bundles over
$X$:
$$(E,\theta)=\bigoplus_{i=1}^{r}(G_i,\theta_i)^{\oplus m_i}$$ such
that the modulo $p$ reductions $\{(G_i,\theta_i)_0\}$s are Higgs
stable and pairwise nonisomorphic, then one has a corresponding
decomposition of $\Z_{p^r}$-representations for a natural number
$r$:
$$
\V\otimes_{\Z_p}\Z_{p^r}=\bigoplus_{i=1}^{r}\V_i
$$
with the equality $\rank_{\Z_{p^r}}\V_i=m_i\rank_{\sO_X}(G_i)$ for
each $i$.
\end{theorem}
The main technical result underlying the above theorem is a $p$-adic
analogue of the Higgs correspondence in the subobjects setting (see
\cite{OV}, \cite{SXZ} for char $p$ case). Combined with the result
of Faltings loc. cit., we obtain a rather satisfactory $p$-adic
analogue of the Simpson correspondence for crystalline
representations of the arithmetic fundamental groups in the
subobjects setting. Precisely, we have the following
\begin{theorem}[Theorem \ref{theorem on one to one correspondence for periodic points}]
Assume $X$ proper smooth over $W$ with connected geometric generic
fiber. Suppose $M\in \mathcal{MF}^\nabla$ non $p$-torsion. Then, for
each natural number $r$, there is a one to one correspondence
between the set of $\Z_{p^r}$-subrepresentations of $\V\otimes
\Z_{p^r}$ and the set of periodic Higgs subbundles in $(E,\theta)$
whose periods are divisors of $r$.
\end{theorem}
Besides its applications, the introduction of the notion of a
periodic Higgs subbundle is another contribution of the paper, which
shall have its importance in the $p$-adic nonabelian Hodge theory.
The first result in this field is the work of Deninger and Werner
\cite{DW}, which gives a $p$-adic analogue of the classical result
of Narasimhan and Seshadri for stable vector bundles and unitary
representations. Using the theory of almost \'{e}tale extensions
developed by himself, Faltings has obtained a vast generalization.
In \cite{Fa3}, he has established in the curve case a correspondence
between Higgs bundles and generalized representations. One of major
open problems in this field is to show semistable Higgs bundles
correspond to genuine representations. In the setting of subobjects,
semistability is hidden in, actually equivalent to, the degree zero
condition. We show that over char $p$ this topological assumption on
Higgs subbundles is equivalent to quasi-periodicity, and periodic
Higgs subbundles are in one to one correspondence to
subrepresentations in $\V\otimes k$ of the arithmetic algebraic
fundamental group. We have also considered the analogous problem
over mixed characteristic and has obtained a partial result.
Therefore, we believe that a proper notion of (quasi-)periodic Higgs
bundles connects the notion of semistable Higgs bundles on the one
hand and genuine representations of geometric fundamental groups on
the other hand. In our recent joint work with Lan \cite{LSZ1}, we have developed this point of view in a general setting over positive characteristic.\\

\textbf{Acknowledgements.} The second named author would like to
thank for the hospitality of School of Mathematical Sciences, the
University of Science and Technology of China. We thank Guitang Lan
for a careful reading of this paper. We thank also Xiaotao Sun for
his comments to this work.

\section{The category $\mathcal{MF}^{\nabla}$}\label{MF category}
For the convenience of the reader, we collect some results due to G.
Faltings on the category $\mathcal{MF}^\nabla_{[0,n]}(X), n\leq p-2$
(see Ch. II \cite{Fa2}, see also \S3-4 \cite{Fa1}). Let $X$ be a
smooth $W$-scheme. For $X=\Spec \ W$, the category was introduced by
Fontaine and Laffaille \cite{FL}, consisting of strong $p$-divisible
filtered Frobenius crystals. In this case, there is no connection
involved. In Ch. II \cite{Fa2}, Faltings generalized the category of
Fontaine and Laffaille to a geometric base $X$ as well as the
comparison theory, which gives an equivalence of categories between
this category and a certain category of \'{e}tale local systems over
$X^0$. We would like also remind the reader that A. Ogus has
developed the category of $F$-$T$ crystals (see \cite{O}) from
another point of view, which is however closely related to the category $\mathcal{MF}^{\nabla}$.\\

The notation $\mathcal{MF}^\nabla_{[0,n]}(X)$ in \cite{Fa2} means
originally for $p$-torsion objects. Here we shall include into
the category non $p$-torsion objects as well, whose any reduction modulo $p^{m+1}, m\geq 0$ is an object of $\mathcal{MF}^\nabla$ (see Ch. II h) \cite{Fa2}).
To our purpose, we state here only the exact definition of a non $p$-torsion object in the category, in a form closer to \S3 \cite{Fa1}. For a $p$-torsion object, one needs to modify the formulation below on the strong $p$-divisibility of the relative Frobenius, which shall cause problems mainly in notations.\\

A small affine subset $U$ of $X$ is an open affine subscheme
$U\subset X$ over $W$ which is \'{e}tale over $\G_m^d$. As $X$ is
smooth over $W$, an open covering $\sU$ consisting of small affine
subsets of $X$ exists. For each $U\in \sU$, we choose a Frobenius
lifting $F_{\hat U}$ on $\hat U$, the $p$-adic completion of $U$. An
object in $\mathcal{MF}_{[0,n]}^{\nabla}(X)$ is a quadruple $(M,
Fil^{\cdot}, \nabla, \Phi)$, where
\begin{itemize}
    \item [i)] $(M,Fil^{\cdot})$ is a locally filtered free $\sO_{X}$-module
    with
    $$
    Fil^0M=M,\quad Fil^{n+1}M=0.
    $$
    \item [ii)] $\nabla$ is an integrable connection on $M$ satisfying the Griffiths
    transversality:
    $$
    \nabla(Fil^iM)\subset Fil^{i-1}M \otimes \Omega_{X|W}.
    $$
    \item [iii)] The valuation $\Phi_{F_{\hat U}}$ of the relative Frobenius over $(U,F_{\hat U})$ is an $\sO_{\hat U}$-linear morphism $\Phi_{F_{\hat U}}: F_{\hat U}^*M\to M$ with the strong
    $p$-divisible property:
    $$\Phi_{F_{\hat U}}(F_{\hat U}^*Fil^iM)\subset p^iM,$$
    and
    $$\sum_{i=0}^{n}\frac{\Phi_{F_{\hat U}}(F_{\hat U}^*Fil^iM)}{p^i}=M.$$
    \item [iv)] $\Phi_{F_{\hat U}}$ is horizontal
    with respect to the connection $F_{\hat U}^*\nabla$ on $F_{\hat U}^*M$ and
    $\nabla$ on $M$.
\end{itemize}
The locally filtered-freeness in i) means that $M$ is locally free
and the filtration $Fil^{\cdot}$ on $M$ is locally split. The
pull-back connection $F_{\hat U}^*\nabla$ on $F_{\hat U}^*M$ is the
composite
\begin{eqnarray*}
   F_{\hat U}^*M=F_{\hat U}^{-1}M\otimes_{F_{\hat U}^{-1}\sO_{\hat U}}\sO_{\hat
U} &\stackrel{F_{\hat U}^{-1}\nabla\otimes id}{\longrightarrow}&
 (F_{\hat U}^{-1}M\otimes F_{\hat U}^{-1}\Omega_{\hat
U})\otimes_{F_{\hat U}^{-1}\sO_{\hat U}}\sO_{\hat U} \\
    &=&  F_{\hat U}^*M\otimes F_{\hat U}^*\Omega_{\hat U} \stackrel{id\otimes
dF_{\hat U}}{\longrightarrow}F_{\hat U}^*M\otimes \Omega_{\hat U}.
\end{eqnarray*}
The horizontal condition iv) is expressed by the commutativity of
the diagram
$$
 \xymatrix{
    F_{\hat U}^*M \ar[d]_{F_{\hat U}^*\nabla} \ar[r]^{\Phi_{F_{\hat U}}} &  M \ar[d]^{\nabla} \\
    F_{\hat U}^*M\otimes
\Omega_{\hat U}\ar[r]^{\Phi_{F_{\hat U}}\otimes id} &  M\otimes
\Omega_{\hat U}. }
$$
We shall explain the Taylor formula relating two valuations of the
relative Frobenius as follows. Write $\hat U=\Spf R$ and $F: R\to R$
the chosen Frobenius lifting. Let $R'$ be another $p$-adically
complete, $p$-torsion free $W$-algebra, equipped with a Frobenius
lifting $F': R'\to R'$ and a morphism of $W$-algebras $\iota: R\to
R'$. Then the valuation $\Phi_{F'}: F'^*(\iota^*M)\to \iota^*M$ of
$\Phi$ over $(R',F')$ is the composite
$$
F'^*\iota^*M\stackrel{\alpha}{\cong}
\iota^*F^*M\stackrel{\iota^*\Phi_{F}}{\longrightarrow} \iota^*M,
$$
where the isomorphism $\alpha$ is given explicitly by the formula
(after choosing a system of \'{e}tale local coordinates
$\{t_1,\cdots,t_d\}$ of $U$):
$$
\alpha(e\otimes 1)=\sum_{\underline i}\nabla_{\partial}^{\underline
i}(e)\otimes \frac{z^{\underline i}}{\underline{i}!}.
$$
Here $\underline{i}=(i_1,\cdots,i_d)$ is a multi-index, and
$z^{\underline i}=z_1^{i_1}\cdots z_d^{i_d}$ with$$z_i=F'\circ
\iota(t_i)-\iota\circ F(t_i), 1\leq i\leq d,$$ and
$$\nabla_{\partial}^{\underline j}=\nabla_{\partial_{
t_1}}^{i_1}\cdots\nabla_{\partial_{t_d}}^{i_d}.$$
\begin{example}\label{geometric situation}
Let $f: Y\to X$ be a proper smooth morphism of relative dimension
$n\leq p-2$ between smooth $W$-schemes. Assume that the relative
Hodge cohomologies $R^if_*\Omega^j_Y, i+j=n$ has no torsion. It
follows from Theorem 6.2 \cite{Fa2} that the crystalline direct
image $R^nf_{*}(\sO_{Y},d)$ is an object in
$\mathcal{MF}^{\nabla}_{[0,n]}(X)$.
\end{example}
We call an object of $\mathcal{MF}^\nabla$ in the above example an
object arising from geometry. The main result about this category is
the following
\begin{theorem}[Theorem 2.6* \cite{Fa2}]\label{Faltings theorem}
Notations as above. There exists a fully faithful contravariant
functor $\mathbf{D}$ from $\mathcal{MF}^\nabla$ to the category of
\'{e}tale local systems over $X^0$. The image is closed under
subobjects and quotients.
\end{theorem}
It is more convenient to use the \emph{covariant} functor
$\mathbf{D}^t$, which for a non $p$-torsion object is defined by
$$
\mathbf{D}^t(M)=\Hom(\mathbf{D}(M),\Z_p)\otimes \Z_p(n),
$$
where $\Z_p(n)$ is the Tate twist (see Ch. II h) \cite{Fa2} for
$p$-torsion objects). The image of the functor $\mathbf{D}^t$ is
called the category of crystalline sheaves over $X^0$. It has an
adjoint functor $\mathbf{E}^t$ from the category of crystalline
sheaves to the
category $\mathcal{MF}^\nabla$ as given in Ch. II f)-g) loc. cit..\\

For a non $p$-torsion object $M\in \mathcal{MF}^\nabla$ and each
$U\in \sU$ with chosen Frobenius lifting $F_{\hat U}$, one defines a
local operator $\tilde \Phi_{F_{\hat U}}$ on the set of subbundles
of $M_{U}$ by
$$
\tilde \Phi_{F_{\hat U}}(M')=\sum_{i=0}^{n}\frac{\Phi_{F_{\hat
U}}}{p^i}F_{\hat U}^*Fil^iM',
$$
where $M'\subset M_{U}$ and $Fil^iM'=Fil^iM\cap M'$. A de Rham
subbundle $(M',\nabla)$ of $(M,\nabla)$ is said to be \emph{$\tilde
\Phi$-stable} if
\begin{itemize}
    \item [(i)] the induced filtration $Fil^\cdot M'$ is filtered free,
    \item [(ii)] it holds for each $U\in \sU$
that
$$
\sum_{i=0}^{n}\frac{\Phi_{F_{\hat U}}}{p^i}F_{\hat U}^*Fil^iM'=M'.
$$
\end{itemize}
It follows from the above Taylor formula that the condition (ii) is
independent of the choices of $F_{\hat U}$s. Also, it is tedious but
straightforward to formulate the notion of $\tilde \Phi$-stable de
Rham subbundles for a $p$-torsion object of $\mathcal{MF}^\nabla$.
A subobject of $M\in \mathcal{MF}^\nabla$ is just
a $\tilde \Phi$-stable de Rham subbundle of $(M,\nabla)$, which by Theorem \ref{Faltings theorem} corresponds to a subrepresentation of $\V$.

\section{Periodic Higgs subbundles in positive
characteristic}\label{section on periodic Higgs subbundles in char
p} Assume $X$ smooth projective over $W$ with connected geometric
generic fiber. Let $M$ be an object in the category
$\mathcal{MF}^\nabla$, $\V$ the corresponding representation and
$(E,\theta)=Gr_{Fil^\cdot}(M,\nabla)$ the associated Higgs bundle.
We have shown previously that $(E,\theta)_0$ is Higgs semistable. In
this section we shall show further that
\begin{theorem}\label{irreducibility implies stability}
The Higgs bundle $(E,\theta)_0$ is Higgs stable iff the
representation $\V\otimes k$ is irreducible.
\end{theorem}
The analogous result over $\C$ is a highly nontrivial result in the
theory of Simpson correspondence between complex local systems and
Higgs bundles \cite{Si}. Motivated by this correspondence, one asks
further a refined version of the above theorem.
\begin{question}\label{motivating question in char p}
Is there a one to one correspondence between the set of
subrepresentations of $\V\otimes k$ and the set of Higgs subbundles
of $(E,\theta)$ with trivial Chern classes?
\end{question}
The aim of this section is to give an answer of this question which
yields the above theorem as a direct consequence. Our answer relies
on the Simpson correspondence in positive characteristic established
by Ogus and Vologodsky \cite{OV}. For simplicity, we assume from now
on that $pM=0$ and $X$ is only smooth proper over $W$, so that the
de Rham bundle $(M,Fil^\cdot,\nabla)$ as well as its associated
Higgs bundle are defined over $X_0/k$. Let $X_0'=X_0\times_{\Spec\
k, F_k}\Spec\ k$, where $F_k: \Spec\ k \to \Spec\ k$ is the absolute
Frobenius. The pull-back of a Higgs module over $X_0$ via the
natural map $X_0'\to X_0$ is a Higgs module over $X_0'$. One
identifies the two categories of Higgs modules over $X_0$ and over
$X_0'$. There is a convenient $W_2$-lifting $X_1'$ of $X_0'$,
obtained from $X_1$ via the base change. Then, by setting
$(\sX,\sS)=(X_0/k, X_1'/W_2(k))$, the inverse Cartier transform
$C_{\sX/\sS}^{-1}$ in loc. cit. associates any Higgs subbundle of
$(E,\theta)$ to a flat subbundle of $(M,\nabla)$\footnote{This
statement requires clarification. To avoid interrupting the
main idea, we postpone this task to the appendix.}. For
simplicity, we denote the inverse Cartier transform in our context by $C_0^{-1}$. \\

Our basic observation is that the operator $Gr_{Fil^\cdot}\circ
C_0^{-1}$ acts on the set of Higgs subbundles, and the action is not
trivial in general. Here is an example.
\begin{example}[Section 7 \cite{SZZ}]
Let $F$ be a totally real field and $D$ a quaternion division
algebra over $F$ which is split at one unique real place $\tau$ of
$F$. Let $K$ be an imaginary quadratic field and $L$ a totally
imaginary quadratic extension of $F$ contained in $D$. Put
$\Phi=\Hom(L,\bar \Q)$. Fix an embedding $\bar \Q\to \bar \Q_p$.
Assume $F$ is unramified at $p$ and each prime of $F$ over $p$ stays
prime in $L$. Let $\mathfrak p$ be the prime of $F$ over $p$ given
by $\tau$. One formulates a moduli functor of PEL type over
$\sO_{F_{\mathfrak p}}$ (see \S5.2 \cite{Car}). Under suitable
conditions, it is fine represented by a smooth projective curve
$\mathcal M$ over $\sO_{F_{\mathfrak p}}$, together with a universal
abelian scheme $f: \sX \to \sM$. Let $(M,Fil^\cdot,\nabla)$ be the
relative de Rham bundle of $f$ and $(E,\theta)$ the associated Higgs
bundle. Then one has an eigen-decomposition under the action of the
commutative subalgebra $\sO_{L\otimes K}$ in the endomorphism
algebra of $f$:
$$
(E,\theta)_0=\bigoplus_{\phi\in \Phi}(E_\phi,\theta_\phi)\oplus
(E_{\bar \phi},\theta_{\bar \phi}).
$$
Theorem 7.3 (1) \cite{SZZ} shows that
$$
Gr_{Fil^\cdot}\circ C_0^{-1}(E_\phi,\theta_\phi)=(E_{\sigma
\phi},\theta_{\sigma \phi}),
$$
where $\sigma$ is the Frobenius element in the Galois group. Thus if
the orbit of a $\phi$ under the $\sigma$-action has more than one
element, it holds that
$$
Gr_{Fil^\cdot}\circ C_0^{-1}(E_\phi,\theta_\phi)\neq
(E_\phi,\theta_\phi).
$$
\end{example}
This example leads us to introduce the following
\begin{definition}[Periodic Higgs subbundles over $k$]\label{periodic Higgs subbundles in char p}
Let $(E,\theta)$ be as above. A Higgs subbundle $(G,\theta)$ in
$(E,\theta)$ is said to be periodic if there is a natural number $r$
such that the equality
$$
(Gr_{Fil^{\cdot}}\circ C_0^{-1})^{r}(G,\theta)=(G,\theta)
$$
holds. The least natural number satisfying the equality is called
the period of $(G,\theta)$.
\end{definition}
Clearly, a periodic Higgs subbundle is a subsystem of Hodge bundles, that is, $$G=\oplus_iG^{i,n-i},\quad \textrm{with}\quad G^{i,n-i}=G\cap E^{i,n-i}.$$ An $(E_\phi,\theta_\phi)$ in the above example, as well as its bar counterpart, is a periodic Higgs subbundle. As the algebraic cycles given by elements of $\sO_{L\otimes K}$ are Tate cycles, it
decomposes the \'{e}tale local system $\V\otimes k$ accordingly. It
is not difficult to arrange a natural one to one correspondence
between eigen-components in $\V\otimes k$ and those in
$(E,\theta)_0$. Naturally, one may wonder if one could deduce the
correspondence without reference to the action of algebraic cycles
but rather relying only on the notion of periodic Higgs subbundles.
Our result shows that it is indeed the case and thereby gives an
answer to the original question.
\begin{theorem}\label{char p periodic case}
Notation as above. Then there exists a one to one correspondence
between the set of $\F_{p^r}$-subrepresentations in $\V\otimes
\F_{p^r}$ and the set of periodic Higgs subbundles in $(E,\theta)$
whose periods are divisors of $r$. Moreover, the
$\sigma$-conjugation on the representation side corresponds to the
operator $Gr_{Fil^\cdot}\circ C_0^{-1}$ on the Higgs side.
\end{theorem}
\begin{remark}\label{remark on quasiperiodic}
A Higgs subbundle $(G,\theta)$ is said to be \emph{quasi-periodic}
if the following equality for a pair $(r,s)$ of integers with
$r>s\geq 0$ holds:
$$
(Gr_{Fil^{\cdot}}\circ C_0^{-1})^r(G,\theta)=(Gr_{Fil^{\cdot}}\circ
C_0^{-1})^s(G,\theta).
$$
It is shown in the proof of Theorem 4.17 (4) in \cite{OV} that for
any nilpotent Higgs bundle $G$ with level less than or equal to
$p-1$,
$$
[C_0^{-1}(G)]=[F_X^*(G)],
$$
where $[\ ]$ denotes the class of a coherent $\sO_X$-module in
$K_0(X)$. The equality implies that a quasi-periodic Higgs subbundle
has trivial Chern classes. Conversely, a Higgs subbundle
$(G,\theta)$ with trivial Chern classes ought to be quasi-periodic.
The reason for it is given in the proof of the next theorem.
Clearly, a periodic Higgs subbundle is quasi-periodic. However, we
are lack of a good criterion to guarantee the injectivity of the
operator $Gr_{Fil^\cdot}\circ C_0^{-1}$ which is equivalent to that
any quasi-periodic Higgs subbundle is indeed periodic.
\end{remark}
Now we proceed to deduce Theorem \ref{irreducibility implies
stability} from Theorem \ref{char p periodic case}.
\begin{proof}
First of all, we draw a simple property for a periodic Higgs
subbundle.
\begin{proposition}\label{perodic implies Higgs semistable of slope
zero} A periodic Higgs subbundle is Higgs semistable of slope zero.
\end{proposition}
\begin{proof}
For a Higgs subbundle $(G,\theta)\subset (E,\theta)$, it follows
from Lemma 3.2 \cite{SXZ} that
$$
\mu(Gr_{Fil^{\cdot}}\circ C_0^{-1}(G,\theta))=p\mu(G).
$$
Therefore, $\mu(G)=0$ because of the periodicity of $G$. By
Proposition 0.2 loc. cit., $(E,\theta)$ is Higgs semistable of slope
zero. Then the statement follows by noting that a Higgs subbundle of
$(G,\theta)$ is also a Higgs subbundle of $(E,\theta)$.
\end{proof}
Now assume first that $(E,\theta)$ is stable. If $\V\otimes k$ was
not irreducible, then $\V\otimes \F_{p^r}$ is reducible for some
$r\in \N$. It follows from Theorem \ref{char p periodic case} that
there is a nontrivial proper periodic Higgs subbundle in
$(E,\theta)$, which contradicts the assumption by the last proposition. Thus
$\V\otimes k$ is irreducible. Conversely, assume $\V\otimes k$ is
irreducible. If $(E,\theta)$ was not stable, then there is a
nontrivial proper Higgs subbundle $(G,\theta)\subset (E,\theta)$ of
slope zero. Note that the operator $Gr_{Fil^\cdot}\circ C_0^{-1}$
does not change the slope, rank and definition field of
$(G,\theta)$. Since there are only finitely many Higgs subbundles of
$(E,\theta)$ with the same slope, rank and definition field as
$(G,\theta)$, $(G,\theta)$ ought to be quasi-periodic. Take a pair
$(r,s)$ for $(G,\theta)$. Then the Higgs subbundle
$(Gr_{Fil^\cdot}\circ C_0^{-1})^s(G,\theta)$ is periodic. By Theorem
\ref{char p periodic case}, it follows that $\V\otimes \F_{p^r}$ is
reducible which contradicts the assumption. This completes the
proof.
\end{proof}
In \S\ref{further applications} one finds more generalizations of
the above result obtained as consequences of Theorem \ref{char p
periodic case}. In the remaining paragraph we shall concentrate on
the proof of Theorem \ref{char p periodic case}. The key is to
notice a basic property of the inverse Cartier transform in the
current setup.
\begin{proposition}\label{basic property of C_0^{-1}}
Let $(G,\theta)$ be a Higgs subbundle of $(E,\theta)$. If
$$
Gr_{Fil^{\cdot}}\circ C_0^{-1}(G,\theta)=(G,\theta)
$$
is satisfied, then $C_0^{-1}(G,\theta)$ is $\tilde \Phi$-stable, and hence corresponds to a subrepresentation of $\V$ by Theorem \ref{Faltings theorem}.
\end{proposition}
\begin{proof}
The question is local. For each small affine $U\subset X$, one can
express $C_0^{-1}(G,\theta)$ locally as follows: take a local basis
$\{g^{i,n-i}\}$ of $G^{i,n-i}$ over $U_0$, and then a set of
elements $\{\tilde g^{i,n-i}\}$ in $M$ with $\tilde g^{i,n-i}\in
Fil^iH$ satisfying the condition
$$
\tilde g^{i,n-i} \mod Fil^{i+1}M=g^{i,n-i}, \ 0\leq i\leq n.
$$
Then it holds that
$$
C_0^{-1}(G,\theta)_{U_0}=\Span[\frac{\Phi_{F_{\hat U}}}{p^i}(F_{\hat
U}^*\tilde g^{i,n-i}), \ 0\leq i\leq n].
$$
The period one property for $(G,\theta)$ implies that we can take
$\{\tilde g^{i,n-i}\}$ to be a local basis of $C_0^{-1}(G,\theta)$.
Thus the $\tilde \Phi$-stability of $C_0^{-1}(G,\theta)$ is just the
local expression of $C_0^{-1}$ as above.
\end{proof}
A direct consequence of the previous proposition is the special case
$r=1$ of Theorem \ref{char p periodic case}.
\begin{corollary}\label{fixed char p case}
Let $\sV$, $\sH$ and $\sE$ be the set of subrepresentations of $\V$,
$\tilde \Phi$-stable de Rham subbundles of $M$ and periodic Higgs
subbundles of $(E,\theta)$ of period one respectively. Then there
are one to one correspondences between the sets:
$$
\sV\xleftrightharpoons[\ \mathbf E^t\ ]{\ \mathbf D^t\ }
 \sH
\xleftrightharpoons[Gr_{Fil^\cdot}]{C_0^{-1}}\sE.
$$
\end{corollary}
\begin{proof}
Theorem \ref{Faltings theorem} settles the correspondence between
$\sV$ and $\sH$. It is to show the correspondence between $\sH$ and
$\sE$. Note first that the locally filtered freeness of $M'\subset
M$ is equivalent to the locally freeness of the grading
$Gr_{Fil^{\cdot}}M'$. If a de Rham subbundle $M'$ is further $\tilde
\Phi$-stable, its grading $Gr_{Fil^{\cdot}}M'$ satisfies the period
one condition:
$$
Gr_{Fil^{\cdot}}\circ
C_0^{-1}(Gr_{Fil^{\cdot}}M')=Gr_{Fil^{\cdot}}(M'),
$$
which follows by taking the grading of the $\tilde \Phi$-stability condition. The converse direction is just Proposition \ref{basic property of C_0^{-1}}.
\end{proof}
The proof of Theorem \ref{char p periodic case} for a general $r$ will be reduced to the above case. The main idea is as follows: for a periodic Higgs subbundle $(G,\theta)$ of period $r$, we embed the Higgs subbundles $\{(Gr_{Fil^\cdot}\circ C_0^{-1})^i(G,\theta)\}_{0\leq i\leq r-1}$ into $(E,\theta)^{\oplus r}$ in a suitable way such that the image is periodic of period one. The above corollary gives us then an $\F_p$-subrepresentation $\W \subset \V^{\oplus r}$. Considering $\F_{p^r}$ as trivial representation and forgetting its $\F_p$-algebra structure, one obtains a natural identification of $\F_p$-representations $$
\V^{\oplus r}=\V\otimes_{\F_p}\F_{p^r}.
$$
With this identification in hand, we show further that $\W\subset \V\otimes_{\F_p}\F_{p^r}$ is indeed stable under the multiplication of elements in $\F_{p^r}$ and hence naturally a $\F_{p^r}$-subrepresentation of $\V\otimes_{\F_p}\F_{p^r}$. We give first several preparatory lemmas.\\

For $r\in \N$, we fix a generator $\xi$ of $\F_{p^r}$ as
$\F_p$-algebra with its minimal polynomial $P(t)$. Recall that
$$\mathbf{D}^t(M,\nabla,Fil^{\cdot},\Phi)=\V.$$ Regarding $\F_{p^r}$ as
trivial representation, it holds clearly that
$$
\mathbf{D}^t((M,\nabla,Fil^{\cdot},\Phi)^{\oplus r})=\V\otimes
_{\F_p}\F_{p^r}.
$$
By using the $\F_p$-basis $\{1,\xi,\cdots,\xi^{r-1}\}$ of
$\F_{p^r}$, we label the $r$ copies of $(M,\nabla,Fil^{\cdot},\Phi)$
as
$$\{(^iH,^i\nabla,^iFil,^i\Phi)\}_{0\leq i\leq r-1}.$$
We observe that the map $s$ on $\V\otimes \F_{p^r}$, induced by
multiplication with $\xi$ on the tensor factor $\F_{p^r}$, is an
endomorphism in the category of crystalline sheaves, and hence by
the equivalence of categories corresponds to an endomorphism
$s_{MF}$ on $M^{\oplus r}$ in the category $\mathcal{MF}^\nabla$ with the
minimal polynomial $P(t)$. As $\F_{p^r}\subset k\subset \sO_{X_0}$,
the endomorphism $s_{MF}$ decomposes $\oplus_{i=0}^{r-1}(^iH)$ into
a direct sum of eigenspaces. We need to describe the
eigen-decomposition in an explicit way which will be applied in our
reduction step. For sake of completeness, the proof of the following
simple lemma in linear algebra is supplied.
\begin{lemma}\label{lemma on elementary number theory}
Let $A\in \GL_r(\F_p)$ be the representation matrix of the
$\F_p$-linear map $m_{\xi}: \F_{p^r}\to \F_{p^r}$ under the basis
$\{1,\xi,\cdots, \xi^{r-1}\}$. Then there is an invertible matrix
$$S=(S_1,S_1^{\sigma},\cdots,S_1^{\sigma^{r-1}})\in \GL_r(\F_{p^r})$$
such that
$$
S^{-1}AS=\mathrm{diag}\{\xi,\xi^{\sigma},
\cdots,\xi^{\sigma^{r-1}}\},
$$
where $S_1$ is the first column vector of $S$ and $\sigma\in
Gal(\Q_{p^r}|\Q_p)$ is the Frobenius element.
\end{lemma}
\begin{proof}
As $P(t)\in \F_p[t]$ splits over $\F_{p^r}$ into the product
$\prod_{i=0}^{r-1}(t-\xi^{\sigma^i})$ of linear factors, one has a
basis of eigenvectors of $$m_{\xi}\otimes id:
\F_{p^r}\otimes_{\F_p}\F_{p^r}\to \F_{p^r}\otimes_{\F_p}\F_{p^r}.$$
Pick an eigenvector $S_1$ to the eigenvalue $\xi$, which namely
satisfies the equality $AS_1=\xi S_1$ holds. Applying $\sigma^i$ on
both sides, one obtains then
$$AS_1^{\sigma^i}=\xi^{\sigma^{i}}S_1^{\sigma^i}.$$ So the matrix
$S=(S_1,\cdots,S_1^{\sigma^{r-1}})$ satisfies
$$
AS=S\mathrm{diag}\{\xi,\xi^{\sigma}, \cdots,\xi^{\sigma^{r-1}}\}.
$$
Note that $\{S_1^{\sigma^i}\}_{0\leq i\leq r-1}$ makes a basis of
eigenvectors of $m_{\xi}\otimes 1$ and hence $S$ is invertible.
\end{proof}

\begin{lemma} \label{lemma on eigendecomposition of s_MF}
The endomorphism $s_{MF}$ decomposes $\oplus_{i=0}^{r-1}\ ^iH$ into
direct sum $\oplus_{i=0}^{r-1}M^i$ of eigenspaces such that
\begin{itemize}
    \item [(i)] one has an explicit isomorphism
$$
(M,\nabla,Fil^{\cdot})\stackrel{\alpha_i}{\cong}
(M^i,\nabla^i:=\nabla^{\oplus r}|_{M^i},Fil^i:=Fil^{\oplus
r}|_{M^i}),
$$
\item [(ii)]$\phi:=\Phi^{\oplus r}$ permutes $\{M^i\}_{0\leq i\leq
r-1}$ cyclically.
\end{itemize}
\end{lemma}
\begin{proof}
Let $\triangle: M\to \oplus_{i=0}^{r-1}\ ^iH$ be the diagonal
embedding with $i$-th component $\triangle^i$, and
$$
S_1=\left(%
\begin{array}{c}
  a_0 \\
  \vdots \\
  a_{r-1} \\
\end{array}%
\right).
$$
It follows from the last lemma that for $0\leq i\leq r-1$,
$$\sum_{j=0}^{r-1}a_j^{\sigma^i}\triangle^j(M)$$ is the
eigenspace of $s_{MF}$ with eigenvalue $\xi^{\sigma^i}$. Clearly,
the isomorphism of vector bundles
$$
\alpha_i=\sum_{j=0}^{r-1}a_j^{\sigma^i}\triangle^j: M\to M^i
$$
respects also the connection and the filtration. Hence (i) follows.
We have (ii) immediately because of the semilinearity of $\phi$.
\end{proof}
Set $(E^i,\theta^i)=Gr_{Fil^i}(M^i,\nabla^i)$. Then $\alpha_i$ in
the above lemma induces the isomorphism of Higgs bundles
$$
\beta_i: (E,\theta)\cong (E^i,\theta^i).
$$
The following easy lemma follows immediately from the semilinearity
of $C_0^{-1}$ and will be applied below.
\begin{lemma}\label{easy lemma for beta}
It holds for a Higgs subbundle $(G,\theta)\subset (E,\theta)$,
$$
Gr_{Fil^{\cdot}}\circ C_0^{-1}(\beta_i(G,\theta))=\beta_{i+1\mod
r}(Gr_{Fil^{\cdot}}\circ C_0^{-1}(G,\theta)), \ 0\leq i\leq r-1
$$
\end{lemma}
Now we proceed to the proof of Theorem \ref{char p periodic case}.
\begin{proof}
Let $\sE^{r}$ be the set of periodic Higgs subbundles of
$(E,\theta)$ whose periods divide $r$, and $\sV^r$ the set of
$\F_{p^r}$-subrepresentations of $\V\otimes \F_{p^r}$. Now we
consider $M^{\oplus r}\in \mathcal{MF}^\nabla$ together with the
endomorphism $s_{MF}$ described as above. Let $\sH^r$ be the set of
$s_{MF}$-invariant $\tilde \phi$-stable de Rham subbundles of
$(M,\nabla,Fil^{\cdot})^{\oplus r}$. We shall show the
correspondences of Corollary \ref{fixed char p case} for $M^{\oplus
r}$ (instead of $M$) induce the claimed correspondence between
$\sV^r$ and $\sE^r$, using $\sH^r$ as a bridge. First of all, an
$\F_{p^r}$-subrepresentation of $\V\otimes \F_{p^r}$ is nothing but
a $\F_p$-subrepresentation of
$$\V\otimes \F_p\{1\}\oplus \cdots\oplus \V\otimes
\F_p\{\xi^{r-1}\}=\V\otimes \F_{p^r}
$$
which is invariant under the endomorphism $s$. Thus the functors
$\D^*$ and $\E^*$ restricts to a one to one correspondence between
$\sV^r$ and $\sH^r$. By Lemma \ref{lemma on eigendecomposition of
s_MF} on the eigen-decomposition of $s_{MF}$, an element of $\sH^r$
is given by a direct sum $\oplus_{i=0}^{r-1}(M^{i'})$ which is
$\tilde \phi$-stable, where $M^{i'}\subset M^i$ is
$\nabla^i$-invariant for each $i$, and since $\tilde \phi$ permutes
the factors $\{M^i\}$s cyclically, $M^{i'}$ is just $\tilde
\phi^{i}(M^{0'})$ for $1\leq i\leq r-1$. Thus an element of
$\sH^{r}$ can be represented by
$$
M^{0'}\oplus \tilde \phi(M^{0'})\oplus \cdots\oplus \tilde
\phi^{r-1}(M^{0'}),
$$
where $M^{0'}\subset M^0$ is a de Rham subbundle satisfying $\tilde
\phi^{r}(M^{0'})=M^{0'}$. It is clear that the functors
$Gr_{Fil^{\cdot}}$ and $C^{-1}_0$ induce a one to one correspondence
between $\sH^r$ and a set of Higgs subbundles of
$(E,\theta)^{\oplus r}=\oplus_{i=0}^{r-1}(E^i,\theta^i)$ of the following
form:
$$
(G,\theta)\oplus Gr_{Fil^\cdot}\circ C_0^{-1}(G,\theta)\oplus
\cdots\oplus (Gr_{Fil^\cdot}\circ C_0^{-1})^{r-1}(G,\theta)
$$
for a Higgs subbundle $(G,\theta)\subset (E^0,\theta^0)$ with the
property
$$
(Gr_{Fil^\cdot}\circ C_0^{-1})^{r}(G,\theta)=(G,\theta).
$$
By Lemma \ref{easy lemma for beta}, the isomorphism $\beta_0^{-1}:
(E^0,\theta^0)\cong (E,\theta)$ induces an identification between
the previous set of Higgs subbundles and $\sE_0^{r}$. Therefore, we
have established a one to one correspondence between $\sV^r$ and
$\sE^r$. Finally, let $\W\subset \V\otimes \F_{p^r}$ be an element
of $\sV^r$. Its $\sigma$-conjugation
$\W^\sigma:=\W\otimes_{\F_{p^r},\sigma}\F_{p^r}$ is considered as an
$\F_{p^r}$-subrepresentation of $\V\otimes \F_{p^r}$ via the natural
isomorphism
$$\V\otimes_{\F_p}\F_{p^r}\otimes_{\F_{p^r},\sigma}\F_{p^r}\cong
\V\otimes_{\F_p}\F_{p^r}, v\otimes \lambda\otimes \mu\mapsto
v\otimes \lambda\mu^\sigma.
$$
So one considers the endomorphism $s^{\sigma}$ on $\V\otimes
\F_{p^r}$ induced by multiplication with $\xi^\sigma$ and its
corresponding endomorphism $s_{MF}^\sigma$ on
$(M,\nabla,Fil^{\cdot},\Phi)^{\oplus r}$. Chasing the proof of Lemma
\ref{lemma on eigendecomposition of s_MF}, one sees that the $i$-th
eigenspace of $s_{MF}^\sigma$ is the $i+1 \mod r$-th eigenspace of
$s_{MF}$. This means that under the above correspondence between
$\sV^{r}$ and $\sH^r$, if $M^{0'}\oplus \cdots \oplus M^{r-1'}$
corresponds to $\W$, then $M^{1'}\oplus \cdots \oplus M^{r-1'}\oplus
M^{0'}$ corresponds to $\W^\sigma$. Therefore if $(G,\theta)\in
\sE^r$ corresponds to $\W$, then $Gr_{Fil^\cdot}\circ
C_0^{-1}(G,\theta)$ corresponds to $\W^\sigma$.
\end{proof}

\section{Periodic Higgs subbundles in mixed characteristic}\label{periodic Higgs subbundles in mixed
char}
We would like to extend our previous results in char $p$ to mixed
characteristic. For a non $p$-torsion object $M\in
\mathcal{MF}^\nabla$ we ask the following question, which is
parallel to Question \ref{motivating question in char p}.
\begin{question}
Is there a one to one correspondence between the set of
subrepresentations of $\V\otimes \hat{\bar{\Z_p}}$ and the set of
Higgs subbundles of $(E,\theta)\otimes \sO_{\bar X}$ with trivial
Chern classes, where $\bar X:=X\times_W \Spec \ \hat{\bar{\Z_p}}$?
\end{question}
This question seems to be much more difficult than the question in
char $p$ case. What we have obtained in this section is a partial
result. The following theorem is the lifted version of Theorem
\ref{char p periodic case}.
\begin{theorem}\label{theorem on one to one correspondence for periodic
points} For each $r\in \N$, there is a one to one correspondence
between the set of $\Z_{p^r}$-subrepresentations of
$\V\otimes_{\Z_p}\Z_{p^r}$ and the set of periodic Higgs subbundles
of $(E,\theta)$ whose periods divide $r$.
\end{theorem}
The meaning of a periodic Higgs subbundle in mixed characteristic
will be explained below. The key to the above theorem is the
construction of a lifting of the inverse Cartier transform to mixed
characteristic. The construction is done in an inductive
way. So we shall consider first the lifting of the inverse Cartier
transform to $W_2$. While the inverse Cartier transform associates to any Higgs subbundle of $(E,\theta)_0$ a de Rham subbundle, the objects of
our lifted inverse Cartier transform over $W_2$ are \emph{not} all Higgs
subbundles of $(E,\theta)_1$, rather those subject to the periodic condition in char $p$. See Proposition \ref{extension of inverse cartier transform to the larger set} for the precise statement. \\

Let $(G,\theta)\subset (E,\theta)_1$ be a subsystem of Hodge bundles. By abuse
of notation, we denote the image of $(G,\theta)_0$ in $(E,\theta)_0$
again by $(G,\theta)_0$. A similar abuse applies to the modulo $p$
reduction of a de Rham subbundle in $(M,\nabla)_1$.
\begin{theorem}\label{C_1^{-1}}
If $(G,\theta)_0\subset (E,\theta)_0$ is a periodic Higgs subbundle
of period one, then one constructs a de Rham subbundle
$C_1^{-1}(G,\theta)\subset (M,\nabla)_1$ with the same rank as $G$
satisfying the equality
$$
(C_1^{-1}(G,\theta))_0=C_0^{-1}(G,\theta)_0.
$$
Furthermore, if the equality
$$
Gr_{Fil^{\cdot}}\circ C_1^{-1}(G,\theta)=(G,\theta)
$$
is satisfied, then $C_1^{-1}(G,\theta)\subset M_1$ is $\tilde
\Phi$-stable, hence corresponds to a subrepresentation of $\V\otimes \Z_p/p^2$ by Theorem \ref{Faltings theorem}.
\end{theorem}
We shall call $C_1^{-1}$ in the theorem an inverse Cartier transform over $W_2$. Its construction is based on our previous work \cite{SXZ} aiming at
a 'physical' understanding of the inverse Cartier transform of Ogus
and Vologodsky. Recently, we have generalized the construction to a
certain category of nilpotent Higgs bundles which have no bearing with the category $\mathcal{MF}^\nabla$. For clarity, we shall carry
out only the local version in this section and complete the global
construction in \S\ref{globalization of lifting of inverse cartier transform}.\\

{\itshape A local inverse Cartier transform over $W_2$.} In the
following paragraph, $X$ is assumed to be affine with a Frobenius
lifting $F_{\hat X}: \hat X\to \hat X$. Let $M\in
\mathcal{MF}^\nabla$ with $pM\neq 0$ and $p^2M=0$. Let
$(G,\theta)\subset (E,\theta)$ be a Higgs subbundle satisfying
$$
Gr_{Fil^{\cdot}}\circ C_0^{-1}(G,\theta)_0=(G,\theta)_0.
$$
Our starting point of the construction is to observe the existence
of special liftings for a basis of $G$.
\begin{lemma}\label{existence of liftings}
Let $\{g^{i,n-i}\}_{0\leq i\leq n}$ be a basis of
$G=\oplus_{i=0}^{n} G^{i,n-i}$. Then there exists a set of elements
$\{\tilde g^{i,n-i}\}_{0\le i\leq n}$ in $M$ with $\tilde
g^{i,n-i}\subset Fil^{i}M$ satisfying two conditions:
\begin{itemize}
    \item [(i)] $\tilde g^{i,n-i}\mod Fil^{i+1}M=g^{i,n-i}$,
    \item [(ii)] $\tilde g^{i,n-i}\mod pM \in
    C_0^{-1}(G,\theta)_0$.
\end{itemize}
\end{lemma}
\begin{proof}
Without loss of generality we can assume in the argument that each
set $g^{i,n-i}$ consists of only one element if nonempty. Put
$$
\tilde g^{n,0}=g^{n,0}.
$$
For each $0\leq i\leq n-1$, we take any $\tilde g^{'i,n-i}\in M$
satisfying (i). We shall modify it as follows: consider its modulo
$p$ reduction $\tilde g_0^{'i,n-i}\in M_0$, which satisfies
$$
\tilde g_0^{'i,n-i} \mod Fil^{i+1}M_0=g^{i,n-i}_0\in G^{i,n-i}_0.
$$
Since
$$
Gr_{Fil^{\cdot}}[C_0^{-1}(G,\theta)_0]=(G,\theta)_0
$$
holds by assumption, there exists a $\breve{g}_0^{i,n-i}\in
Fil^{i}C_0^{-1}(G,\theta)_0$ satisfying
$$
\breve{g}_0^{i,n-i}\mod Fil^{i+1}M_0=g^{i,n-i}_0.
$$
In other words, it holds that
$$
\omega_0^{i+1,n-i-1}:=\tilde g^{'i,n-i}_0-\breve{g}_0^{i,n-i}\in
Fil^{i+1}M_0.
$$
Now we pick an $\omega^{i+1,n-i-1}\in Fil^{i+1}M$ lifting
$\omega_0^{i+1,n-i-1}$ and set
$$
\tilde g^{i,n-i}=\tilde g^{'i,n-i}-\omega^{i+1,n-i-1}.
$$
Then this modified element satisfies both conditions.
\end{proof}
\begin{proposition and definition}\label{local cartier}
Notation as above. Then the $\sO_{X_1}$-submodule
$$
\Span[\frac{\Phi_{F_{\hat X}}}{p^{i}}(F_{\hat X}^*\tilde
g^{i,n-i}),\ 0\leq i\leq n]
$$
is a well-defined de Rham submodule of $(M,\nabla)$. More precisely,
the span is independent of the choice of basis elements
$\{g^{i,n-i}\}$ of $G$ and the choice of liftings $\{\tilde
g^{i,n-i}\}$ in Lemma \ref{existence of liftings}. We call it
$C_1^{-1}(G,\theta)$.
\end{proposition and definition}
\begin{remark}
In \S\ref{globalization of lifting of inverse cartier transform}, we
show further that $C_1^{-1}(G,\theta)$ is also independent of the
choice of Frobenius lifting $F_{\hat X}$.
\end{remark}
\begin{proof}
For simplicity, we omit the subscript of $\Phi_{F_{\hat X}}$ and
denote $e\otimes 1$ for the pullback of an element $e\in M$ via
$F_{\hat X}$.
\begin{claim}\label{independece of choice of liftings}
The span is independent of the choice of elements $\{\tilde g^{i,
n-i}\}$ in Lemma \ref{existence of liftings} as well as the choice
of basis elements $\{g^{i,n-i}\}$ of $G$. Hence
$C_1^{-1}(G,\theta)\subset M$ is well defined.
\end{claim}
\begin{proof}
Let $\{\tilde g^{'i,n-i}\}$ be another set of elements in Lemma
\ref{existence of liftings}. Then by condition (i) we can write
$$\tilde g^{'i,n-i}=\tilde g^{i,n-i}+\omega^{i+1,n-i-1},$$ for an
$\omega^{i+1,n-i-1}\in Fil^{i+1}M$. By condition (ii),
$$
\omega^{i+1,n-i-1}_0\in Fil^{i+1}C_0^{-1}(G,\theta)_0.
$$
Note that the two conditions of Lemma \ref{existence of liftings} imply the equality:
$$
Fil^{i+1}C_0^{-1}(G,\theta)_0=\Span[\tilde g^{n,0}_0,\cdots,\tilde g_0^{i+1,n-i-1}].
$$
It follows that
$$
\omega^{i+1,n-i-1}\in pFil^{i+1}M+\Span[\tilde
g^{n,0},\cdots,\tilde g^{i+1,n-i-1}].
$$
Now we use the induction on $i$ to show the claim: for $i=n$, there is nothing to prove. The induction hypothesis for $i+1$ means that
$$
\Span[\frac{\Phi}{p^j}(\tilde g^{j,n-j}\otimes
1),\ i+1\leq j\leq n] =\Span[\frac{\Phi}{p^j}(\tilde
g^{'j,n-j}\otimes 1),\ i+1\leq j\leq n].
$$
It follows from the above discussion that
\begin{eqnarray*}
  \frac{\Phi}{p^i}(\tilde g^{'i,n-i}\otimes 1)-\frac{\Phi}{p^i}(\tilde g^{i,n-i}\otimes 1) &=&  \frac{\Phi}{p^i}(\omega^{i+1,n-i-1}\otimes 1) \\
  \in\frac{\Phi}{p^i}F_{\hat X}^*[pFil^{i+1}M  & +&  \Span[\tilde
g^{n,0},\cdots,\tilde
g^{i+1,n-i-1}]] \\
    &=& \frac{\Phi}{p^i}F_{\hat X}^*[\Span[\tilde g^{n,0},\cdots,\tilde
g^{i+1,n-i-1}]].
\end{eqnarray*}
The last equality follows from the fact $\frac{\Phi}{p^{i}}(pFil^{i+1}M)=0$. Since clearly
$$
\frac{\Phi}{p^i}F_{\hat X}^*[\Span[\tilde g^{n,0},\cdots,\tilde
g^{i+1,n-i-1}]]\subset \Span[\frac{\Phi}{p^j}(\tilde
g^{j,n-j}\otimes 1),\ i+1\leq j\leq n],
$$
the case for $i$ then follows. The $i=0$ case is the first statement
of the claim. Note that two different choice of bases of $G$ are
interrelated through an invertible matrix. It relates also special
liftings in Lemma \ref{existence of liftings} for these two bases.
Clearly, they define the same span. Thus $C_1^{-1}(G,\theta)$ is a
well defined submodule of $M$.
\end{proof}
It remains to show the following
\begin{claim}
The $\sO_{X_1}$-submodule $C_1^{-1}(G,\theta)\subset M$ is invariant
under the action of $\nabla$. Hence $C_1^{-1}(G,\theta)$, equipped
with the induced connection, is a de Rham submodule of $(M,\nabla)$.
\end{claim}
\begin{proof}
Without loss of generality, we assume $X$ to be a curve. Take a local coordinate $t$ of $X$, i.e. $\Omega_{X|M}=\sO_X\{dt\}$, and set $\partial=\frac{d}{dt}$. It suffices to show that for each $0\leq i\leq n$,
$$
\nabla_{\partial}(\frac{\Phi}{p^{i}}(\tilde
g^{i,n-i}\otimes 1))\in C_{1}^{-1}(G,\theta).
$$
The horizontal property of $\Phi$ can be explicitly expressed by
$$
\nabla_{\partial}[\frac{\Phi}{p^i}(\tilde
g^{i,n-i}\otimes 1)]=\frac{\Phi}{p^{i-1}}[\nabla_{\partial}(\tilde g^{i,n-i})\otimes a],
$$
for an $a\in \sO_{X_1}$. We shall show that
$$
\frac{\Phi}{p^{i-1}}[\nabla_{\partial}(\tilde g^{i,n-i})\otimes 1]\in
C_{1}^{-1}(G,\theta).
$$
The assumption that $G\subset E$ is a Higgs subbundle means that $G$
is invariant under the action of
$\theta_{\partial}=Gr_{Fil^{\cdot}}\nabla_{\partial}$. Thus one
finds a unique $b\in \sO_{X_1}$ such that
$$
\nabla_{\partial}(\tilde g^{i,n-i}) \mod Fil^iM=bg^{i-1,n-i+1}.
$$
It follows then
$$
\omega^{i,n-i}:=\nabla_{\partial}(\tilde g^{i,n-i})-b\tilde g^{i-1,n-i+1}\in Fil^{i}M.
$$
As clearly
$$
\frac{\Phi}{p^{i-1}}(b\tilde g^{i-1,n-i+1}\otimes 1)\in C_1^{-1}(G,\theta),
$$
it suffices to show
$$
\frac{\Phi}{p^{i-1}}(\omega^{i,n-i}\otimes 1)\in C_1^{-1}(G,\theta).
$$
As $\frac{\Phi}{p^{i-1}}(\omega^{i,n-i}\otimes 1)\in pM$, it is equivalent to show
$$
\frac{\Phi}{p^{i}}(\omega_0^{i,n-i}\otimes 1)\in [C_1^{-1}(G,\theta)]_0=C_0^{-1}(G,\theta)_0.
$$
The equivalence is clear from Lemma \ref{simple lineare algebra lemma}, which is elementary. Now that
$$
\omega_0^{i,n-i}=\nabla_{\partial}(\tilde g_0^{i,n-i})-b_0\tilde g_0^{i-1,n-i+1},
$$
with $\tilde g_0^{i,n-i}\in C_0^{-1}(G,\theta)_0$ and $C_0^{-1}(G,\theta)_0$ is $\nabla$-invariant, one has
$$
\omega_0^{i,n-i}\in C_0^{-1}(G,\theta)_0.
$$
Finally, because $(G,\theta)_0$ is periodic of period one,
$$
\frac{\Phi}{p^{i}}(\omega_0^{i,n-i}\otimes 1)\in C_0^{-1}(G,\theta)_0
$$
by Proposition \ref{basic property of C_0^{-1}}. This shows the claim.
\end{proof}
\end{proof}
Now we are going to show Theorem \ref{C_1^{-1}} by assuming the global existence of $C_1^{-1}$.
\begin{proof}
It follows from the strong $p$-divisibility of $\Phi$ that
$C_1^{-1}(G,\theta)$ is free $\sO_{X_1}$-module of the same rank as
$G$. Note that the set of elements $\{g_0^{i,n-i}\}$ is a basis of
$G_0$. Then we have
$$
(C_1^{-1}(G,\theta))_0= \Span_{\sO_{X_0}}[\frac{\Phi_{F_{\hat
X}}}{p^{i}}(\tilde g_0^{i,n-i}\otimes 1),\ 0\leq i\leq n],
$$
which is exactly $C_{0}^{-1}(G,\theta)_0$ by its very construction.
Now assume furthermore
$$
Gr_{Fil^{\cdot}}\circ C_1^{-1}(G,\theta)=(G,\theta).
$$
Then we can take $\{\tilde g^{i,n-i}\}$ of Lemma \ref{existence of
liftings} to be a basis of $C_1^{-1}(G,\theta)$, and the $\tilde
\Phi$-stability of $C_1^{-1}(G,\theta)$ is just the definition of
$C_1^{-1}$.
\end{proof}
The assumption of $(G,\theta)$ for the existence of
$C_1^{-1}(G,\theta)$ can be relaxed. In fact, using the same
technique in the reduction step from a general period to the period
one case, we can show the following
\begin{proposition}\label{extension of inverse cartier transform to the larger set}
Let $(G,\theta)\subset (E,\theta)_1$ be a Higgs subbundle whose
modulo $p$ reduction is a periodic Higgs subbundle in
$(E,\theta)_0$. Then there exists $C_1^{-1}(G,\theta)\subset
(M,\nabla)_1$ with the same rank as $G$ satisfying the equality
$$
(C_1^{-1}(G,\theta))_0=C_0^{-1}(G,\theta)_0.
$$
\end{proposition}
\begin{remark}
In the above proposition as well as Theorem \ref{C_1^{-1}},
we have assumed that $(G,\theta)\subset (E,\theta)_1$ to be a subsystem of Hodge bundles. However, this assumption is not necessary. In fact, $C_1^{-1}(G,\theta)$ exists for \emph{any} Higgs subbundle of $(E,\theta)_1$ with the periodic condition in char $p$.
\end{remark}
The detail of the proof is postponed to \S\ref{globalization
of lifting of inverse cartier transform}, because it is
more urgent to note that we are already in an inductive situation: \\

For a non $p$-torsion $M\in \mathcal{MF}^\nabla$, we define
inductively the set of periodic Higgs subbundles of $(E,\theta)_m$
and a lifting of the inverse Cartier transform $C_{m+1}^{-1}$ over
$W_{m+1}:=W_{m+1}(k)$, where $m$ runs from zero to infinity. Based on the
inverse Cartier transform $C_0^{-1}$ of Ogus and Vologodsky, we have
defined previously the set of periodic Higgs subbundles in
$(E,\theta)_0$. The last proposition asserts that the inverse
Cartier transform lifts to an operator over $W_2$ for those Higgs
subbundles of $(E,\theta)_1$ whose modulo $p$ reduction are
periodic. Thus we make the following
\begin{definition}[Periodic Higgs subbundles over $W_2$]
Notation as above. A Higgs subbundle $(G,\theta)$ of $(E,\theta)_1$
is periodic if there two natural numbers $r_0,r_1$ such that
\begin{itemize}
    \item [(1)] $(Gr_{Fil^{\cdot}}\circ
    C_0^{-1})^{r_0}(G,\theta)_0=(G,\theta)_0$ and then,
    \item [(2)] $(Gr_{Fil^{\cdot}}\circ
    C_1^{-1})^{r_1}(G,\theta)=(G,\theta)$ hold.
\end{itemize}
\end{definition}
Via a direct generalization of the construction of $C_1^{-1}$, one
defines a further lifting $C_2^{-1}$ over $W_3$ for the set of Higgs
subbundles of $(E,\theta)_2$ whose modulo $p^2$ reduction are
periodic in the above sense, and then the set of periodic Higgs
subbundles of $(E,\theta)_2$, and so on. This process culminates
with the following
\begin{definition}[Periodic Higgs subbundles over $W_{m+1}$ and $W$]\label{definition of periodic Higgs subbundle at
formal level} A Higgs subbundle $(G,\theta)\subset (E,\theta)_m$ is
periodic if there are a sequence of natural numbers $\{r_i\}_{0\leq
i\leq m}$ such that inductively from $i=0$ to $i=m$ the equality
$$
(Gr_{Fil^{\cdot}}\circ C_{i}^{-1})^{r_i}(G,\theta)_i=(G,\theta)_i
$$
holds. A Higgs subbundle of $(E,\theta)$ is periodic if its
reduction in $(E,\theta)_m$ is periodic for all $m\geq 0$.
\end{definition}
For a periodic Higgs subbundle $(G,\theta)$, we list the periods of
$(G,\theta)_m$ into a sequence of natural numbers $r_0,r_1,\cdots$.
Clearly, $r_i$ divides $r_j$ for $i>j$. Since the numbers of Higgs
subbundles in $(E,\theta)_m$ are bounded by a constant independent
of $m$, the above sequence is stable after finitely many terms. Thus
a Higgs subbundle $(G,\theta)$ is periodic iff there exists a
natural number $r$ such that
$$
(Gr_{Fil^{\cdot}}\circ C_{m}^{-1})^{r}(G,\theta)_m=(G,\theta)_m, \
\forall m\geq 0.
$$
Now we come to the proof of Theorem \ref{theorem on one to one
correspondence for periodic points}.
\begin{proof}
Our lifting $C_{m+1}^{-1}$ of the inverse Cartier transform lifts
the basic property of $C_0^{-1}$ as given in Proposition \ref{basic
property of C_0^{-1}}. For $m=0$, this is a part of statements in
Theorem \ref{C_1^{-1}}, and its proof generalizes directly to a
general $m$. Using the same argument as in Corollary \ref{fixed char
p case}, one shows the one to one correspondence between the set of
subrepresentations of $\V$ and the set of periodic Higgs subbundles
of $(E,\theta)$ of period one by identifying both with the set of
$\tilde \Phi$-stable de Rham subbundles of $M$. To show the general
case, we note first that Lemma \ref{lemma on elementary number
theory} and its consequent lemmas hold over $\Z_{p^r}$. So the
reduction step to the period one case as carried in the proof of
Theorem \ref{char p periodic case} can be applied to the mixed
characteristic situation as well. This shows the theorem.
\end{proof}

\section{Further applications}\label{further applications}
In this section, $X$ is assumed to be smooth projective over $W$ throughout. Notations as before. We start with a direct consequence of Theorem \ref{char p periodic case}.
\begin{proposition}\label{semisimple implies polystability}
If $\V\otimes k$ is semisimple, then $(E,\theta)_0$ is polystable.
\end{proposition}
\begin{proof}
The assumption implies the decomposition
$$
\V\otimes \F_{p^r}=\oplus_i\V_i
$$
into direct sum of $\F_{p^r}$-representations whose direct factors are all absolutely irreducible.
It is clear that the correspondence in Theorem \ref{char p periodic case} respects direct sum. So we obtain a corresponding decomposition
$$
(E,\theta)_0=\oplus_i(G_i,\theta_i)
$$
into direct sum of periodic Higgs subbundles. Each factor ought to be stable, since the corresponding factor is
otherwise not absolutely irreducible by a similar argument in Theorem \ref{irreducibility implies stability}.
\end{proof}
We would like to have the converse statement of the above proposition. What we have obtained is a conditional result.
The proof of the following lemma is identical to that for a semistable bundle of degree
zero, which is standard.
\begin{lemma}\label{JH filtration}
The following statements hold for $(E,\theta)_0$:
\begin{itemize}
    \item [(i)] there is a filtration $0=(G_0,\theta_0)\subset (G_1, \theta_1)\subset
\cdots\subset (G_r,\theta_r)=(E,\theta)_0$ by Higgs subbundles such
that the quotient $(G_i,\theta_i)/(G_{i-1},\theta_{i-1})$ is Higgs
stable and $\deg (G_i)=0$ for $1\leq i\leq r$,
    \item [(ii)] if $
0=(G'_0,\theta'_0)\subset (G'_1, \theta'_1)\subset \cdots\subset
(G'_s,\theta'_s)=(E,\theta)_0$ is another filtration enjoying the
properties of (i), then $r=s$ and there exists a permutation $\tau$
of $\{1,\cdots,r\}$ such that
$(G_i,\theta_i)/(G_{i-1},\theta_{i-1})$ is isomorphic to
$(G'_{\tau(i)},\theta'_{\tau(i)})/(G'_{\tau(i)-1},\theta_{\tau(i)-1})$.
\end{itemize}
\end{lemma}
A filtration in (i) is called a Jordan-H\"{o}lder (abbreviated as JH) filtration of $(E,\theta)_0$. Put
$$
gr(E,\theta)_0=\oplus_{i=1}^{r}(G_i,\theta_i)/(G_{i-1},\theta_{i-1}),
$$
which is independent of the choice of a JH filtration. The number $r$ in the above expression is said to be the length of a JH filtration on $(E,\theta)_0$.
\begin{assumption}\label{nonisomorphic components in the grading}
Assume the factors in $gr(E,\theta)_0$ are non-isomorphic to each other.
\end{assumption}
\begin{proposition}\label{partial converse}
Assume \ref{nonisomorphic components in the grading}. If $(E,\theta)_0$ is polystable, then $\V\otimes k$ is semisimple.
\end{proposition}
Note that the operator $Gr_{Fil^{\cdot}}\circ C_0^{-1}$ does not
commute with direct sum in general, although $C_0^{-1}$ does. This
makes the problem subtle. We observe however the following property
of the operator, which implies the result under the assumption. If
we know that the operator is injective, then we can even remove the
assumption.
\begin{proposition}\label{the operator preserves the irreducibles}
Let $(G,\theta)\subset (E,\theta)_0$ be a Higgs subbundle of
degree zero. Then $Gr_{Fil^{\cdot}}\circ C_0^{-1}(G,\theta)$ is Higgs stable iff $(G,\theta)$ is Higgs stable.
\end{proposition}
We derive Proposition \ref{partial converse} first.
\begin{proof}
As $(E,\theta)_0$ is polystable, we write
$$
(E,\theta)_0=\oplus_i(G_i,\theta_i)
$$
into a direct sum of stable factors. Because of Proposition \ref{the
operator preserves the irreducibles}, $Gr_{Fil^{\cdot}}\circ
C_0^{-1}(G_i,\theta_i)$ is again stable of degree zero for each $i$.
For a chosen $i_0$, we consider the composite
$$
Gr_{Fil^{\cdot}}\circ C_0^{-1}(G_{i_0},\theta_{i_0})\subset
(E,\theta)_0\twoheadrightarrow (G_i,\theta_i).
$$
It is either zero or an isomorphism because of stability. By the assumption, there is a unique $j_0$ such that the composite onto $(G_{j_0},\theta_{j_0})$ is an isomorphism. It follows that
$$
Gr_{Fil^{\cdot}}\circ
C_0^{-1}(G_{i_0},\theta_{i_0})=(G_{j_0},\theta_{j_0}),
$$
and that the operator induces a map on the set of indices of direct factors. This map must be injective and hence bijective: assume the contrary and say
$$
Gr_{Fil^{\cdot}}\circ
C_0^{-1}(G_{1},\theta_{1})=Gr_{Fil^{\cdot}}\circ
C_0^{-1}(G_{2},\theta_{2}).
$$
For
$$
M_i:=C_0^{-1}(G_i,\theta_i), \quad Fil^{\cdot}_i:=Fil^{\cdot}M_i, \
i=1,2,
$$
it holds that $M_1\cap M_2=0$, and then the previous equality implies inductively
$$
Fil_1^n=Fil_2^n=0,\ Fil_1^{n-1}=Fil_2^{n-1}=0,\cdots, Fil^{0}_1=Fil^0_2=0,
$$
which is absurd. Therefore, each direct stable factor is periodic and by Theorem \ref{char p periodic case}, $\V\otimes k$ is a direct sum of irreducible representations, i.e, semisimple.
\end{proof}
One direction of Proposition \ref{the operator preserves the
irreducibles} is clear. Namely, if $Gr_{Fil^{\cdot}}\circ
C_0^{-1}(G,\theta)$ is stable, then $(G,\theta)$ must be stable. To
show the converse direction, we need a lemma.
\begin{lemma}\label{the operator preserves JH filtration}
Let $0=(G_0,\theta_0)\subset (G_1, \theta_1)\subset \cdots\subset
(G_r,\theta_r)=(E,\theta)_0$ be a JH filtration of
$(E,\theta)_0$. Then $$0=Gr_{Fil^{\cdot}}\circ
C_0^{-1}(G_0,\theta_0)\subset Gr_{Fil^{\cdot}}\circ C_0^{-1}(G_1,
\theta_1)\subset \cdots\subset Gr_{Fil^{\cdot}}\circ
C_0^{-1}(G_r,\theta_r)$$ is again a JH filtration of $(E,\theta)_0$.
\end{lemma}
\begin{proof}
Note first that $Gr_{Fil^{\cdot}}\circ C_0^{-1}(G_i,\theta_i)$ is
Higgs semistable of degree zero. So the grading
$Gr_{Fil^{\cdot}}\circ C_0^{-1}(G_i,\theta_i)/Gr_{Fil^{\cdot}}\circ
C_0^{-1}(G_{i-1},\theta_{i-1})$ for each $i$ is Higgs semistable of
degree zero. It is to show that each grading is in fact Higgs
stable. For $1\leq i\leq r$ let
$$\pi_i: Gr_{Fil^{\cdot}}\circ C_0^{-1}(G_i,\theta_i)\twoheadrightarrow
Gr_{Fil^{\cdot}}\circ C_0^{-1}(G_i,\theta_i)/Gr_{Fil^{\cdot}}\circ
C_0^{-1}(G_{i-1},\theta_{i-1})$$ be the natural projection. If this
grading was not Higgs stable, then by Lemma \ref{JH filtration} (i),
there is a nontrivial JH filtration of this grading. Then the
preimage of this JH filtration in $Gr_{Fil^{\cdot}}\circ
C_0^{-1}(G_i,\theta_i)$ via $\pi_i$ is a nontrivial refinement of
the inclusion $Gr_{Fil^{\cdot}}\circ
C_0^{-1}(G_{i-1},\theta_{i-1})\subset Gr_{Fil^{\cdot}}\circ
C_0^{-1}(G_i,\theta_i)$ whose gradings are by construction Higgs
stable of degree zero. Therefore we will obtain a new JH filtration
of $(E,\theta)_0$ with strictly greater length, which contradicts
Lemma \ref{JH filtration} (ii).
\end{proof}
Then Proposition \ref{the operator preserves the irreducibles} is shown as follows.
\begin{proof}
One can refine the inclusion $(G,\theta)\subset (E,\theta)_0$ into a
JH filtration $$0\subset (G,\theta)=(G_1,\theta_1)\subset
\cdots\subset (E,\theta)_0.$$ Then the previous lemma
shows that $$0\subset
Gr_{Fil^{\cdot}}\circ C_0^{-1}(G,\theta)\subset \cdots \subset
(E,\theta)_0$$ is again a JH filtration. In particular,
$Gr_{Fil^{\cdot}}\circ C_0^{-1}(G,\theta)$ is Higgs stable of degree
zero.
\end{proof}
A composition series for $\V\otimes k$ is a filtration of subrepresentations whose gradings are irreducible.
The length of a composition series is the number of the irreducibles in its grading. It follows from Schur's lemma
that two composition series have the same length. A natural question is to compare
this length on the representation side with that on this Higgs side. The next result generalizes Theorem \ref{irreducibility implies stability}.
\begin{proposition}\label{existence of JH filtration with periodic components}
The length of a composition series of $\V\otimes k$ is less than or equal to the length of a JH filtration of $(E,\theta)_0$. Assume \ref{nonisomorphic components in the grading}. Then they are equal.
\end{proposition}
\begin{proof}
By Theorem \ref{char p periodic case}, a composition series on $\V\otimes k$ gives rise to
a filtration on $(E,\theta)_0$ whose constituents are periodic Higgs subbundles, which are
Higgs semistable of degree zero by Proposition \ref{perodic implies Higgs semistable of slope zero}. Thus the first statement is clear. To get the second statement, it suffices to produce a
JH filtration on $(E,\theta)_0$ consisting of periodic Higgs subbundles. We start with an arbitrary JH filtration $$0=(G'_0,\theta'_0)\subset
(G'_1,\theta'_1)\subset \cdots\subset
(G'_r,\theta'_r)=(E,\theta)_0.$$ For each $l\in \N$ and $1\leq i\leq
r$ we write the grading
$$(Gr_{Fil^{\cdot}}\circ
C_0^{-1})^l(G'_i,\theta'_i)/(Gr_{Fil^{\cdot}}\circ
C_0^{-1})^l(G'_{i-1},\theta'_{i-1})$$ by $gr_i^l$ for short. It
follows from Lemma \ref{JH filtration} (ii) and Lemma \ref{the
operator preserves JH filtration} that $gr_i^l$ has its isomorphism
class in the set $\{gr_1,\cdots,gr_{r}\}$. So for $i=r-1$, there
exist nonnegative integers $r\geq n_{r-1}>m_{r-1}\geq 0$ such that
$gr_{r-1}^{n_{r-1}}\cong gr_{r-1}^{m_{r-1}}$. Thus, because of the assumption, the composite
$$
(Gr_{Fil^{\cdot}}\circ
C_0^{-1})^{n_{r-1}}(G'_{r-1},\theta'_{r-1})\subset
E_0\twoheadrightarrow E_0/(Gr_{Fil^{\cdot}}\circ
C_0^{-1})^{m_{r-1}}(G'_{r-1},\theta'_{r-1})=gr_{r-1}^{m_{r-1}}
$$
is zero. Hence the above inclusion factorizes through
$$(Gr_{Fil^{\cdot}}\circ
C_0^{-1})^{n_{r-1}}(G'_{r-1},\theta'_{r-1})\subset
(Gr_{Fil^{\cdot}}\circ
C_0^{-1})^{m_{r-1}}(G'_{r-1},\theta'_{r-1}).$$ As they have the same rank and both are of degree zero, the inclusion is actually an equality
$$(Gr_{Fil^{\cdot}}\circ
C_0^{-1})^{n_{r-1}}(G'_{r-1},\theta'_{r-1})= (Gr_{Fil^{\cdot}}\circ
C_0^{-1})^{m_{r-1}}(G'_{r-1},\theta'_{r-1}).
$$
Then we replace the starting JH filtration with the one after the $m_{r-1}$-iterated $Gr_{Fil^{\cdot}}\circ C_0^{-1}$-action. Put
$$
(G_{r-1},\theta_{r-1})=(Gr_{Fil^{\cdot}}\circ C_0^{-1})^{m_{r-1}}(G'_{r-1},\theta'_{r-1}),
$$
which is periodic with period $\leq n_{r-1}-m_{r-1}\leq r$, and denote the obtained JH filtration by
$$0=(G'_0,\theta'_0)\subset (G'_1,\theta'_1)\subset
\cdots\subset(G_{r-1},\theta_{r-1})\subset (E,\theta)_0.$$ Note that
$(G'_i,\theta'_i)$ may differ from the original one. Next, we shall
apply the same argument to the filtration $$0=(G'_0,\theta'_0)\subset
(G'_1,\theta_1)\subset \cdots\subset(G'_{r-1},\theta'_{r-1}).$$ But we shall take the number of iterations of the operator $Gr_{Fil^{\cdot}}\circ
C_0^{-1}$ to be a multiple of $(n_{r-1}-m_{r-1})$. It yields
then nonnegative integers $r-1\geq n_{r-2}>m_{r-2}\geq 0$ such that
$$
(Gr_{Fil^{\cdot}}\circ
C_0^{-1})^{(n_{r-1}-m_{r-1})n_{r-2}}(G'_{r-2},\theta'_{r-2})=
(Gr_{Fil^{\cdot}}\circ
C_0^{-1})^{(n_{r-1}-m_{r-1})m_{r-2}}(G'_{r-2},\theta'_{r-2})
$$
holds. Put $$(G_{r-2},\theta_{r-2})=(Gr_{Fil^{\cdot}}\circ
C_0^{-1})^{(n_{r-1}-m_{r-1})m_{r-2}}(G'_{r-2},\theta'_{r-2}).$$ Then
continue the argument. In the end, we will obtain a JH filtration
whose constituents $(G_i,\theta_i), 1\leq i\leq r-1$ are all periodic.
\end{proof}
The following corollary is immediate after the above discussions.
\begin{corollary}\label{direct corollary in char p}
Assume \ref{nonisomorphic components in the grading}. If $(E,\theta)_0$ decomposes
$$
(E,\theta)_0=\bigoplus_{i=1}^{r}(G_i,\theta_i)
$$
into a direct sum of stable factors, then
$$
\V\otimes k=\bigoplus_{i=1}^{r}\V_i
$$
with $\V_i$ is irreducible and $\dim_k\V_i=\rank_{\sO_{X_0}}G_i$.
\end{corollary}
Our next result replaces the assumption \ref{nonisomorphic components in the grading} with a geometric one.
\begin{proposition}\label{geometric case in char p}
Let $M\in \mathcal{MF}^{\nabla}$ be an object arising from geometry. Suppose that $X_0(k)$ contains an ordinary point (over which the Newton polygon coincides with the Hodge polygon). If
$$
(E,\theta)_0=\bigoplus_{i=1}^r(G_i,\theta_i)^{\oplus
m_i}
$$ decomposes into a direct sum such that each $(G_i,\theta_i)$ is Higgs stable and
$$
(G_i,\theta_i)\ncong(G_j,\theta_j), \ i\neq j,
$$
then accordingly
$$
\V\otimes k=\bigoplus_{i=1}^{r}\V_i
$$
with $\dim_k\V_i=\rank_{\sO_{X_0}}G_i$.
\end{proposition}
\begin{proof}
In the argument we shall use $F_{hod}$ to mean the Hodge filtration
$Fil^\cdot$ and $F_{con}$ the conjugate filtration on $M_0$. We
shall prove only the weight one case, because the higher weight case
is entirely the same. Forgetting the multiplicities $\{m_i\}$, we
rewrite the decomposition of $(E,\theta)_0$ into
$$(E,\theta)_0=\oplus_{i'}(G_{i'},\theta_{i'}).$$ As $C_0^{-1}$
respects direct sum, it yields the decomposition
$$
(M,\nabla)_0=\oplus_{i'}(M_{i'},\nabla_{i'})
$$
with $(M_{i'},\nabla_{i'})=C_0^{-1}(G_{i'},\theta_{i'})$. If follows
from the Cartier-Katz descent (see e.g. Proposition 2.11 \cite{SXZ})
that $F_{con}$ on $M_0$ decomposes accordingly. That is,
$$
F_{con}=\oplus_{i'}M_i'\cap F_{con}.
$$
Now Proposition \ref{the operator preserves the irreducibles}
implies that the images of $(G_{i'},\theta_{i'})$ and
$(G_{j'},\theta_{j'})$ under $Gr_{Fil^\cdot}\circ C_0^{-1}$
intersect trivially or coincide. Thus we can take the following
argument over the \emph{generic point} of $X_0$. So we are
considering vector spaces and their linear maps over the function
field of $X_0$. For simplicity, we shall not change the notations
for this base change. The ordinary assumption is equivalent to that
over the generic point the two filtrations $F_{hod}$ and $F_{con}$
on $M_0$ are complementary. We claim that $F_{hod}$ on $M_0$
decomposes accordingly. For that, we consider the composite $\pi$ of
natural morphisms in the following diagram:
$$
\xymatrix{
 & F_{con} \ar[d]_{\hookrightarrow} \ar[dr]^{\pi}        \\
 F_{hod}\ar[r]_{\hookrightarrow}& M_0 \ar[r]_{\twoheadrightarrow}  & M_0/F_{hod}          }
$$
It is an isomorphism by the ordinary assumption. So the
decomposition of $F_{con}$ induces via $\pi$ the following
decomposition:
$$
M_0/F_{hod}=\oplus_{i'}\pi(M_{i'}\cap F_{con}).
$$
Because $\pi$ maps $M_{i'}\cap F_{con}$ into $M_{i'}/M_{i'}\cap
F_{hod}$ and any two of $M_{i'}/M_{i'}\cap F_{hod}$s either
intersect trivially or coincide, it follows that any two of them
intersect trivially and
$$
\pi(M_{i'}\cap F_{con})=M_{i'}/M_{i'}\cap
F_{hod}.
$$
It follows that
$$
M_0/F_{hod}=\oplus_{i'}M_{i'}/M_{i'}\cap F_{hod},
$$
which implies that $F_{hod}=\oplus_{i'}M_{i'}\cap F_{hod}$ as
claimed. Therefore, the following equality holds:
$$
\oplus_{i'}(G_{i'},\theta_{i'})=Gr_{Fil^\cdot}\circ
C_0^{-1}[\oplus_{i'}(G_{i'},\theta_{i'})]=
\oplus_{i'}Gr_{Fil^\cdot}\circ C_0^{-1}(G_{i'},\theta_{i'}).
$$
In particular, we have
$$
\oplus_i(G_i,\theta_i)^{\oplus m_i}=\oplus_iGr_{Fil^\cdot}\circ
C_0^{-1}[(G_i,\theta_i)^{\oplus m_i}].
$$
As for any given $i$ there is a unique $j$ such that
$$Gr_{Fil^\cdot}\circ C_0^{-1}(G_i,\theta_i)\cong (G_j,\theta_j)$$
holds, the previous equality implies that $m_i=m_j$ and
$$
Gr_{Fil^\cdot}\circ C_0^{-1}[(G_i,\theta_i)^{\oplus
m_i}]=(G_j,\theta_j)^{\oplus m_j}.
$$
So the operator $Gr_{Fil^\cdot}\circ C_0^{-1}$ permutes the factors
$\{(G_i,\theta_i)^{\oplus m_i}\}$ and therefore each
$(G_i,\theta_i)^{\oplus m_i}$ is periodic. The result follows from
Theorem \ref{char p periodic case}.
\end{proof}
The following lemma allows us to lift many results in char $p$ to mixed characteristic.
\begin{lemma}\label{unique lifting lemma}
Let $(G,\theta)$ be a Higgs subbundle of $(E,\theta)_m$ satisfying
the following two conditions:
\begin{itemize}
    \item [(i)] the quotient $(E,\theta)_m/(G,\theta)$ is a locally free $\sO_{X_m}$-module,
  \item [(ii)] the set $\Hom_{\sO_{X_0}}((G,\theta)_0,(E,\theta)_0/(G,\theta)_0)$ of morphisms of Higgs bundles is trivial.
\end{itemize}
Then if $(G,\theta)'\subset (E,\theta)_m$ is a Higgs subbundle
with the same rank as $G$ and its modulo $p$ reduction $(G,\theta)'_0$ is equal to $(G,\theta)_0$, then one concludes that $$(G,\theta)'=(G,\theta).$$
\end{lemma}
\begin{proof}
The assumption (i) implies that $G$ has no $p$-torsion as a local basis element  and hence the modulo $p$ reduction map $G\otimes\F_p \to E_0$ is injective. As $G'$ has the same rank as $G$ and the same modulo $p$ reduction, the modulo $p$
reduction map $G'\otimes \F_p \to E_0$ is also injective. For $m=0$, there
is nothing to prove. So we begin with the modulo $p^2$ reductions of
$(G,\theta)'$ and $(G,\theta)$. Denote the composite
$(G,\theta)'_1\subset (E,\theta)_1\twoheadrightarrow
(E,\theta)_1/(G,\theta)_1$ by $\tau$. By the condition (ii), the
modulo $p$ reduction $\tau_0: (G,\theta)'_0=(G,\theta)_0\to
(E,\theta)_0/(G,\theta)_0$ is zero. So $$\tau((G,\theta)'_1)\subset
p[(E,\theta)_1/(G,\theta)_1].$$ Then we continue to consider the
composite $$(G,\theta)'_1\stackrel{\tau}{\to}
p[(E,\theta)_1/(G,\theta)_1]\stackrel{1/p}{\cong}
(E,\theta)_0/(G,\theta)_0.
$$
It descends clearly to a morphism $\tau/p: (G,\theta)_0 \to
(E,\theta)_0/(G,\theta)_0$ and hence is zero for the same reason.
This means that $\tau$ is zero and therefore an inclusion
$$(G,\theta)'_1\subset (G,\theta)_1.$$ Since they have the same rank
and the same modulo $p$ reduction, they are equal. An easy induction
on $m$ shows the lemma.
\end{proof}
The following result uses the full strength of our theory, which therefore can be regarded as our best approximation to our original question in the introduction.
\begin{theorem}\label{lifting statement of prop in geo case}
Let $M\in \mathcal{MF}^{\nabla}$ be a non $p$-torsion object with $\V$ and $(E,\theta)$ as before. Assume one of the following two situations:
\begin{itemize}
    \item [(i)] $(E,\theta)=\bigoplus_{i=1}^{r}(G_i,\theta_i)$ and the reductions mod $p$ of $\{(G_i,\theta_i)\}_{1\leq i\leq r}$ are Higgs stable and nonisomorphic to each other.
    \item [(ii)] $M$ arises from geometry and $(E,\theta)=\bigoplus_{i=1}^{r}(G_i,\theta_i)^{\oplus m_i}$ and the reductions mod $p$ of $\{(G_i,\theta_i)\}_{1\leq i\leq r}$ are Higgs stable and nonisomorphic to each other.
\end{itemize}
Then it holds accordingly that for a natural number $r$,
$$
\V\otimes_{\Z_p} \Z_{p^r}=\bigoplus_{i=1}^{r}\V_i
$$
with $\rank_{\Z_{p^r}}\V_i=\rank_{\sO_X}(G_i)$ in Case (i) and
$$\V\otimes_{\Z_p} \Z_{p^r}=\oplus_i \V_i$$ with
$\rank_{\Z_{p^r}}\V_i=m_i\rank_{\sO_X}(G_i)$ in Case (ii).
\end{theorem}
\begin{proof}
We show only Case (i) since the proof for Case (ii) is entirely similar. Each $(G_i,\theta_i)_0$ is a periodic Higgs subbundle by the assumption. Let $r_i$ be its period. Thus we can apply $C_1^{-1}$ to $(G_i,\theta_i)_1$ by
Proposition \ref{extension of inverse cartier transform to the
larger set}. Then by construction, $(Gr_{Fil^{\cdot}}\circ C_1^{-1})^{r_i}(G_i,\theta_i)_1$ has the same rank as $(G_i)_1$ and its reduction mod $p$ is equal to
$(Gr_{Fil^{\cdot}}\circ C_0^{-1})^{r_i}(G_i,\theta_i)_0$, which is equal to $(G_i,\theta_i)_0 $ by periodicity. So Lemma \ref{unique lifting lemma} applies, and we get the equality
$$
(Gr_{Fil^{\cdot}}\circ C_1^{-1})^{r_i}(G_i,\theta_i)_1=(G_i,\theta_i)_1,
$$
which means that $(G_i,\theta_i)_1\subset (E,\theta)_1$ is periodic with period $r_i$. We continue the argument, and show inductively that $(G_i,\theta_i)_m\subset (E,\theta)_m$ is periodic of period $r_i$ for
any $m\geq 0$. Theorem \ref{theorem on one to one
correspondence for periodic points} (and its proof) implies the corresponding direct decomposition of $\V$ with the claimed properties after tensoring with $\Z_{p^r}$ with $r$ is the least common multiple of all $r_i$s.
\end{proof}
We conclude this section by pointing out a connection with the notion of strongly semistable vector bundles, which was introduced in \cite{LS} and
has played a central role in the work of Deninger and Werner \cite{DW}. This connection becomes more evident in our recent work \cite{LSZ1}.
\begin{proposition}\label{periodic Higgs subbundle with zero Higgs field in char p} Let $(G^{i,n-i},0)\subset (E,\theta)_0$ be a
periodic Higgs subbundle of period $r$. If $(Gr_{Fil^{\cdot}}\circ
C_0^{-1})^j(G^{i,n-i})$ is of pure type for each $1\leq j\leq r$,
then $G^{i,n-i}$ is \'{e}tale trivializable and particularly
strongly semistable.
\end{proposition}
\begin{proof}
We shall show an isomorphism
$$
F_{X_0}^{*r}G^{i,n-i}\cong G^{i,n-i},
$$
which implies that $G^{i,n-i}$ is \'{e}tale trivializable by Satz
1.4 \cite{LS}. To show that, we consider first
$Gr_{Fil^{\cdot}}\circ C_0^{-1}(G^{i,n-i},0)$. As it is of pure type
by assumption, it is isomorphic to $C_0^{-1}(G^{i,n-i},0)$, which is
isomorphic to $F_{X_0}^*G^{i,n-i}$ by Remark 2.2 \cite{OV} (see also
Proposition 2.9 \cite{SXZ}). Note that we can continue the argument
and show inductively for $1\leq j\leq r$ an isomorphism
$$
(Gr_{Fil^{\cdot}}\circ C_0^{-1})^j(G^{i,n-i},0)\cong F_{X_0}^{*j}G^{i,n-i}.
$$
Then the claimed isomorphism follows from the periodicity and the $j=r$ case.
\end{proof}
\begin{remark}\label{remark on necessariness of pure type}
The assumption on purity of the Hodge type of
$$(Gr_{Fil^{\cdot}}\circ C_0^{-1})^j(G^{i,n-i}), \ 1\leq j\leq r$$ made
for the strong semistability of $G^{i,n-i}$ is necessary. The
example in Proposition 6.6 (ii) \cite{SZZ} shows that a certain
power of the operator $Gr_{Fil^{\cdot}}\circ C_0^{-1}$ can turn a
Higgs subbundle with zero Higgs field into a Higgs subbundle with maximal Higgs field, which is not semistable but Higgs semistable.
\end{remark}

\section{A global inverse Cartier transform over $W_2$}\label{globalization of lifting of inverse cartier
transform} Our aim of this section is to globalize the construction
of the local inverse Cartier transform over $W_2$ in \S\ref{periodic
Higgs subbundles in mixed char}. In particular, the proof of
Theorem \ref{C_1^{-1}} is completed in this section. The main
technique underlying the construction is an extensive use of the
Taylor formula (see \S\ref{MF category}). Thus, without loss of
generality, we can assume $X$ to be a curve. This assumption amounts
to simplify a multi-index into a usual index in the arguments. Our
strategy is as follows: firstly we show that the local inverse
Cartier transform over $W_2$ does not depend on the choice of
Frobenius liftings. Secondly, we modify the arguments to show that
the local constructions glue into a global one. Finally, we adapt
the technique of \S\ref{section on periodic Higgs subbundles in char
p} to show that the local inverse Cartier transform over $W_2$
extends over to those Higgs subbundles whose reductions are periodic
in char $p$.\\

Assume $X$ affine with a Frobenius lifting $F_{\hat X}$, and $M\in
\mathcal{MF}^\nabla$ satisfies $p^2M=0$ and $pM\neq 0$. The
following simple lemma reduces certain issues over $W_2$ to char
$p$.
\begin{lemma}\label{simple lineare algebra lemma}
Let $M'\subset M$ be a free $\sO_{X}$-submodule, and $M'_0\subset
M_0$ the image of its modulo $p$ reduction in $M_0$. Then the
isomorphism $\frac{1}{p}: pM\cong M_0$ restricts to an isomorphism
$pM'\cong M'_0$.
\end{lemma}
\begin{proof}
Fix an $\sO_{X_1}$-basis $e$ (resp. $e'$) of $M$ (resp. $M'$). Under
these bases, the inclusion $M'\hookrightarrow M$ is given by a
matrix $A=(a_{ij})$ with $a_{ij}\in \sO_{X_1}$. Then the image
$M'_0$ under the map
$$A_0:=A \mod p: M' \mod p \to M\mod p=M_0$$ is generated by
$A_0\cdot e_0$. On the other hand, $pM'\subset pM$ is generated by
$(pA)\cdot e$. So its image under the isomorphism $\frac{1}{p}$ is
generated also by $A_0\cdot e_0$.
\end{proof}
\begin{proposition}\label{independence of choice of Frobenius lifting}
Let $(G,\theta)$ be as given in Proposition and Definition
\ref{local cartier}. Then $C_{1}^{-1}(G,\theta)$ is independent of
the choice of Frobenius liftings over $\hat X$.
\end{proposition}
\begin{proof}
Take another Frobenius lifting $F'_{\hat X}$. For short, we write
$\phi=\Phi_{F_{\hat X}}$ and $\phi'=\Phi_{F'_{\hat X}}$. Because of
the symmetric roles of $\phi$ and $\phi'$, it suffices to show for
each $0\leq i \leq n$,
$$
\frac{\phi'}{p^{i}}(\tilde g^{i,n-i}\otimes
1)-\frac{\phi}{p^{i}}(\tilde g^{i,n-i}\otimes 1)\in
\Span[\frac{\phi}{p^{i}}(\tilde g^{i,n-i}\otimes 1),\ 0\leq i\leq
n].
$$
We show this by an explicit calculation via the Taylor formula.
Choose a local coordinate $t\in \sO_X$ of $X$. Then the Taylor formula says that
$$
(\phi'-\phi)(e\otimes
1)=\sum_{j=1}^{\infty}\phi(\nabla_{\partial}^j(e)\otimes 1)\otimes
z^j/j!,
$$
where $\partial=\frac{d}{dt}$, $e$ is an element in $M$, and
$z=F'_{\hat X}(t)-F_{\hat X}(t)\in \sO_{\hat X}$ which is divisible
by $p$. For an $e\in M$, the above formula then reads by modulo
$p^2$. So in this case the right hand side of the above formula is
just a finite sum. By the Griffiths transversality,
$$\nabla^{i-j}_{\partial}(\tilde g^{i,n-i})\in Fil^jM, \ 0\leq
j\leq i-1.$$ As $i\leq n\leq p-2$, the above formula for
$e=g^{i,n-i}\in M$ can be written into
$$
\frac{\phi'}{p^i}(\tilde g^{i,n-i}\otimes 1)-\frac{\phi}{p^i}(\tilde
g^{i,n-i}\otimes 1)=I+II,
$$
with
$$
I=\sum_{j=0}^{i-1}\frac{\phi}{p^{j}}(\nabla^{i-j}_{\partial}(\tilde
g^{i,n-i})\otimes 1)\otimes \frac{z^{i-j}}{p^{i-j}(i-j)!}
$$
and
$$
II=\sum_{j\geq i+1}\phi(\nabla_{\partial}^{j}(\tilde
g^{i,n-i})\otimes 1)\otimes \frac{z^{j}}{p^{i}j!}.
$$
We are going to show the terms $I$ and $II$ belong to
$\Span[\frac{\phi}{p^{i}}(\tilde g^{i,n-i}\otimes 1),\ 0\leq i\leq
n]$. Consider first the term $II$. Note that $\frac{z^{j}}{p^{i}j!}$
is divisible by $p$ for $j\geq i+1$. So $II\in pM$. By Lemma
\ref{simple lineare algebra lemma}, $II\in
p\Span[\frac{\phi}{p^{i}}(\tilde g^{i,n-i}\otimes 1),\ 0\leq i\leq
n]$ iff
$$
II/p\in \Span[\frac{\phi}{p^{i}}(\tilde g^{i,n-i}\otimes
1),\ 0\leq i\leq n]_0=C_{0}^{-1}(G,\theta)_0.
$$
So we consider the modulo $p$ reduction of
$\phi(\nabla_{\partial}^{j}(\tilde g^{i,n-i})\otimes 1)$. By the
condition (ii) of Lemma \ref{existence of liftings}, $\tilde
g_0^{i,n-i}\in C_0^{-1}(G,\theta)_0$. As it is $\nabla$-invariant,
it follows that for any $j$,
$$
\nabla_{\partial}^{j}(\tilde g_0^{i,n-i})\in C_0^{-1}(G,\theta)_0.
$$
By Proposition \ref{basic property of C_0^{-1}}, if follows further
that
$$
\phi(\nabla_{\partial}^{j}(\tilde g_0^{i,n-i})\otimes 1)\in
C_{0}^{-1}(G,\theta)_0.
$$
Therefore,
$$II\in p\Span[\frac{\phi}{p^{i}}(\tilde g^{i,n-i}\otimes
1),\ 0\leq i\leq n]\subset \Span[\frac{\phi}{p^{i}}(\tilde
g^{i,n-i}\otimes 1),\ 0\leq i\leq n].$$ Consider next the term $I$.
As $G\subset E$ is $\theta$-invariant, there exists a unique $b_j\in
\sO_{X_1}$ such that
$$
\nabla^{i-j}_{\partial}(\tilde g^{i,n-i})\mod
Fil^{j+1}M=b_jg^{j,n-j}\in G^{j,n-j}.
$$
As clearly $b_j\tilde g^{j,n-j}\mod Fil^{j+1}M=b_jg^{i,n-j}$, it
follows that $$\omega^{j+1,n-j-1}:=\nabla^{i-j}_{\partial}(\tilde
g^{i,n-i})-b_j\tilde g^{j,n-j}\in Fil^{j+1}M.$$ Note that
$\frac{\phi}{p^j}(\omega^{j+1,n-j-1}\otimes 1)\in pM$. Again by
Lemma \ref{simple lineare algebra lemma}, in order to show
$$\frac{\phi}{p^j}(\omega^{j+1,n-j-1}\otimes 1)\in
p\Span[\frac{\phi}{p^{i}}(\tilde g^{i,n-i}\otimes 1),\ 0\leq i\leq
n],$$ it suffices to show
$\frac{\phi}{p^{j+1}}(\omega^{j+1,n-j-1}_0\otimes 1)\in
C_{0}^{-1}(G,\theta)_0$. But this is rather clear, because
$$
\omega_0^{j+1,n-j-1}=\nabla^{i-j}_{\partial}(\tilde
g_0^{i,n-i})-b_{j,0}\tilde g_0^{j,n-j}
$$
belongs to $C_{0}^{-1}(G,\theta)_0$ and by Proposition \ref{basic
property of C_0^{-1}}
$$
\frac{\phi}{p^{j+1}}(F_{U_0}^*C_{0}^{-1}(G,\theta)_0)\subset
C_{0}^{-1}(G,\theta)_0.
$$
As clearly $$\frac{ \phi}{p^j}(b_j\tilde g^{j,n-j}\otimes 1)\in
\Span[\frac{\phi}{p^{i}}(\tilde g^{i,n-i}\otimes 1),\ 0\leq i\leq
n],$$ we have also shown $$I\in \Span[\frac{\phi}{p^{i}}(\tilde
g^{i,n-i}\otimes 1),\ 0\leq i\leq n].$$ This completes the proof.
\end{proof}
From now on, $X$ is assumed to be proper smooth over $W$. Let $\sU$
be a small open affine covering of $X$, together with a choice of
Frobenius lifting $F_{\hat U}$ over $\hat U$ for each $U\in \sU$.
Thus for a Higgs subbundle $(G,\theta)$ in the situation of
Theorem \ref{C_1^{-1}}, we have constructed a set of local de
Rham subbundles $\{C_1^{-1}(G,\theta)|_{U}\}_{U\in \sU}$ with local
properties listed ibid. In order to show $C_1^{-1}(G,\theta)$
exists, it suffices to show the following equality of subbundles in
$M|_{U_1\cap V_1}$ for any $U,V\in \sU$:
$$
[C_{1}^{-1}(G,\theta)|_{U_1}]|_{U_1\cap
V_1}=C_1^{-1}[(G,\theta)|_{U_1\cap V_1}].
$$
Its proof modifies the previous one. Take Frobenius liftings
$F_{\hat U}, F_{\hat V}, F_{\widehat{U\cap V}}$ on $\hat U,\hat
V,\widehat{U\cap U}$ respectively and write $$z=F_{\hat U}\circ
\iota(t)-\iota\circ F_{\widehat{U\cap V}}(t),$$ where
$\iota:\widehat{U\cap V}\hookrightarrow \hat U$ is the natural
inclusion. Then the difference $$ \iota_1^*[\frac{\Phi_{F_{\hat
U}}}{p^i}(\tilde g^{i,n-i}\otimes 1)]-\frac{\Phi_{F_{\widehat{U\cap
V}}}}{p^i}[\iota_1^*(\tilde g^{i,n-i}\otimes 1)]$$ is again
expressed by the Taylor formula. Thus the previous proof carries
over, and it shows that
$$
[C_{1}^{-1}(G,\theta)|_{U_1}]|_{U_1\cap V_1}\subset
C_1^{-1}[(G,\theta)|_{U_1\cap V_1}].
$$
In order to obtain the equality rather than an inclusion, we shall
examine the proof of Proposition \ref{independence of choice of
Frobenius lifting}. Consider first the above difference for $i=0$.
One sees from the proof that the difference belongs to
$p\Phi_{F_{\widehat{U\cap V}}}(\iota_1^*\tilde g^{0,n})$. So it
holds that
$$
\iota_1^*[\Phi_{F_{\hat U}}(\tilde g^{0,n}\otimes
1)]=\Phi_{F_{\widehat{U\cap V}}}[\iota_1^*(\tilde g^{0,n}\otimes
1)].
$$
For a general $1\leq i\leq n$, we shall use induction on $i$. Assume
the truth of the equality for $i-1$, namely,
$$
\Span[\iota_1^*[\frac{\Phi_{F_{\hat U}}}{p^{j}}(\tilde
g^{j,n-j}\otimes 1)],\ 0\leq j\leq i-1]
=\Span[\frac{\Phi_{F_{\widehat{U\cap V}}}}{p^{j}}[\iota_1^*(\tilde
g^{j,n-j}\otimes 1)],\ 0\leq j\leq i-1].
$$
As one sees from the proof that the difference
$$\iota_1^*[\frac{\Phi_{F_{\hat U}}}{p^i}(\tilde g^{i,n-i}\otimes
1)]-\frac{\Phi_{F_{\widehat{U\cap V}}}}{p^i}[\iota_1^*(\tilde
g^{i,n-i}\otimes 1)]$$ belongs to
$$p\frac{\Phi_{F_{\widehat{U\cap V}}}}{p^i}[\iota_1^*(\tilde
g^{i,n-i}\otimes 1)]+\Span[\frac{\Phi_{F_{\widehat{U\cap
V}}}}{p^{j}}[\iota_1^*(\tilde g^{j,n-j}\otimes 1)], 0\leq j\leq
i-1],$$ one obtains the equality also for $i$. So the local
subbundles $\{C_1^{-1}(G,\theta)|_{U}\}_{U\in \sU}$ glue into a
global subbundle $C_1^{-1}(G,\theta)$ of $(M,\nabla)$ as claimed.
Now we proceed to the proof of Proposition \ref{extension of inverse
cartier transform to the larger set}.
\begin{proof}
Let $(G,\theta)\subset (E,\theta)_1$ be a Higgs subbundle with the
equality in char $p$:
$$
(Gr_{Fil^{\cdot}}\circ C_0^{-1})^r(G,\theta)_0=(G,\theta)_0.
$$
As remarked in the proof of Theorem \ref{theorem on one to one
correspondence for periodic points}, Lemma \ref{lemma on elementary
number theory} and its consequent lemmas extend to $W_2$. So we have
the eigen-decomposition
$$
(E,\theta)^{\oplus r}_1=\bigoplus_{i=0}^{r-1}(E^i,\theta^i),
$$
and isomorphisms of Higgs bundles
$$
\beta_i: (E,\theta)_1\cong (E^i,\theta^i).
$$
By Lemma \ref{easy lemma for beta}, the Higgs subbundle
$$
\bigoplus_{i=0}^{r-1}\beta_i[(Gr_{Fil^\cdot}\circ
C_0^{-1})^i(G,\theta)_0]\subset
\bigoplus_{i=0}^{r-1}(E^i,\theta^i)_0
$$
is periodic of period one. So one might be able to reduce the
construction to Theorem \ref{C_1^{-1}}. But this does not quite
succeed. This is because the existence of a Higgs subbundle in
$\bigoplus_{i=0}^{r-1}(E^i,\theta^i)$, whose reduction modulo $p$ is
$\bigoplus_{i=0}^{r-1}\beta_i[(Gr_{Fil^\cdot}\circ
C_0^{-1})^i(G,\theta)_0]$, is not part of our assumption. Instead,
we shall modify our original local construction suitably so that the
previous arguments carries over. We can assume in the following
argument that $r=2$. This assumption does not affect much the proof
for a general $r$, but will simplify the notations greatly. \\

Firstly, the proof of Lemma \ref{existence of liftings} shows that
for each $U\in \sU$, there exists a set of elements $\{\tilde
g^{i,n-i}\}\subset M_1|_{U_1}$ such that
\begin{itemize}
    \item [(i)] $\tilde
g^{i,n-i} \mod Fil^{i+1}M_1=g^{i,n-i}$,
    \item [(ii)] $\tilde
g^{i,n-i}\mod p\in C_0^{-1}[Gr_{Fil^\cdot}\circ
C_0^{-1}(G,\theta)_0]$.
\end{itemize}
Then we define as before
$$
C_{1}^{-1}(G,\theta)|_{U_1}=\Span[\frac{\Phi_{F_{\hat U}}}{p^i}(\tilde g^{i,n-i}\otimes 1),\ 0\leq i\leq n].
$$
It is direct to check the following equalites:
$$
[C_{1}^{-1}(G,\theta)|_{U_1}]_0=[C_0^{-1}(G,\theta)_0]|_{U_0},
$$
and
$$
[\tilde
\Phi(C_{1}^{-1}(G,\theta)|_{U_1})]_0=[C_0^{-1}(Gr_{Fil^\cdot}\circ
C_0^{-1}(G,\theta)_0)]|_{U_0}.
$$
We use now the eigen-decomposition and isomorphisms in the lifted
Lemma \ref{lemma on eigendecomposition of s_MF} over $W_2$:
$$
(M,\nabla)^{\oplus 2}=(M^0,\nabla^0)\oplus (M^1,\nabla^1), \
\alpha_i: (M,\nabla)\cong (M^i,\nabla^i), i=0,1.
$$
We claim that the local subbundles
$$
\{\alpha_0[\tilde \Phi (C_{1}^{-1}(G,\theta)|_{U_1})]\oplus
\alpha_1[C_{1}^{-1}(G,\theta)|_{U_1}]\}_{U\in \sU}
$$
of $(M^0,\nabla^0)\oplus (M^1,\nabla^1)$ are well defined, i.e.
independent of the choices of elements $\{\tilde g^{i,n-i}\}$ given
as above, and $\nabla$-invariant, and they glue. Our old proofs for
$C_1^{-1}$ go through, because the mod $p$ reductions of these local
bundles are simply
$$
\{\alpha_0[C_0^{-1}(Gr_{Fil^{\cdot}}\circ
C_0^{-1}(G,\theta)_0)]\oplus\alpha_1[C_0^{-1}(G,\theta)_0]\}|_{U_0},
$$
and obviously they glue into
$$
\alpha_0[C_0^{-1}(Gr_{Fil^{\cdot}}\circ
C_0^{-1}(G,\theta)_0)]\oplus\alpha_1[C_0^{-1}(G,\theta)_0],
$$
which is just
$$
C_0^{-1}[\beta_0(G,\theta)_0\oplus
\beta_1(Gr_{Fil^{\cdot}}\circ C_0^{-1}(G,\theta)_0)],
$$
and the Higgs bundle $\beta_0(G,\theta)_0\oplus
\beta_1(Gr_{Fil^{\cdot}}\circ C_0^{-1}(G,\theta)_0)$ is periodic of
period one. The gluing implies that
$\{\alpha_1[C_{1}^{-1}(G,\theta)|_{U_1}]\}_{U\in \sU}$ glue into a
de Rham subbundle of $M^1$. Therefore,
$\{C_{1}^{-1}(G,\theta)|_{U_1}\}_{U\in \sU}$ glue into a de Rham
subbundle $C_1^{-1}(G,\theta)$ of $(M,\nabla)_1$ whose modulo $p$
reduction is $C_0^{-1}(G,\theta)_0$, as claimed.
\end{proof}

\section{Appendix: the inverse Cartier transform of Ogus and Vologodsky}
The appendix explains the equivalence (up to sign) of the inverse
Cartier transform of Ogus and Vologodsky \cite{OV} and the
association defined in \cite{SXZ} in the subobjects setting. We thank
heartily Arthur Ogus for pointing out the equivalence follows from
Remark 2.10 \cite{OV} in \cite{O1}. Our exposition is based on his
remark. In the following we shall quote the notations and results in
\cite{OV}, \cite{SXZ} and \cite{LSZ} freely. In his forthcoming
doctor thesis \cite{X}, H. Xin shall explain the equivalence as well as that in the logarithmic case in detail. We divide the proof into two steps:\\

{\itshape Step 1.} Let $(E,\theta)$ be a nilpotent Higgs bundle of exponent $\leq p-1$. The inverse Cartier transform of $(E,\theta)$ after Ogus and Vologodsky is defined by
$$
(M,\nabla)_{(E,\theta)}:=\sB_{\sX/\sS}\otimes_{\hat{\Gamma}_{\cdot}T_{X'/S}}\iota^*\pi^*(E,\theta).
$$
The construction is global. In \cite{LSZ}, we construct a flat
bundle $(M_{\exp},\nabla_{\exp})_{(E,\theta)}$ by gluing the local
flat subbundles
$$
\{F_{U_0}^*E,\nabla_{can}+(id\otimes \frac{dF_{\hat U}}{p})\circ F_{U_0}^*\theta\}_{U\in \sU}
$$
via an exponential function, where $F_{U_0}$ is the absolute Frobenius over $U_0$.
\begin{claim}
There is a functorial isomorphism
$$
(M,\nabla)_{(E,\theta)}\cong (M_{\exp},\nabla_{exp})_{(E,-\theta)}.
$$
\end{claim}
\begin{proof}
Recall that we have chosen an affine covering $\sU$ of $X$, together with a choice of
Frobenius liftings for each $U\in \sU$. Thus over each $U$, the lifting $F_{\hat U}$ defines an isomorphism of $\hat{\Gamma}_{\cdot}T_{U_0'/k}$-modules:
$$
\sB_{\sX/\sS}|_{U_0}\cong F_{U_0/k}^*\hat{\Gamma}_{\cdot}T_{U_0'/k},
$$
where $F_{U_0/k}: U_0\to U_0'$ is the relative Frobenius. Therefore, one has a natural isomorphism
\begin{eqnarray*}
 [(M,\nabla)_{(E,\theta)}]|_{U_0}&\cong & \sB_{\sX/\sS}|_{U_0}\otimes_{\hat{\Gamma}_{\cdot}T_{U_0'/k}}\iota^*\pi^*E|_{U_0}\\
  &\cong& F_{U_0/k}^*\hat{\Gamma}_{\cdot}T_{U_0'/k}\otimes_{\hat{\Gamma}_{\cdot}T_{U_0'/k}}\iota^*\pi^*E|_{U_0} \\
  &\cong&F_{U_0}^*\iota_{*}E.
\end{eqnarray*}
This gives actually an isomorphism of flat bundles
$$
[(M,\nabla)_{(E,\theta)}]|_{U_0}\cong
[(M_{\exp},\nabla_{\exp})_{(E,-\theta)}]_{U_0},
$$
followed from the description in Formula (2.11.2) \cite{OV}. The
sign comes from the involution $\iota$. Now Remark 2.10 loc. cit.
tells how the above isomorphisms change when we choose another
Frobenius lifting. Precisely, let $F_{\hat U}'$ be another choice,
then their difference defines an element $\xi\in
F_{U_0/k}^*T_{U_0'/k}$. As $\sB_{\sX/\sS}$ is a
$F^*_{X_0/k}\hat{\Gamma}_{\cdot}T_{X_0'/k}$-torsor, different local
trivializations are related via the Taylor formula or equivalently
the exponential $\exp\ D_{\xi}$. As $E$ is nilpotent with exponent
$\leq p-1$ by assumption, the action on the Higgs field becomes a
twist via the usual exponential function. When we interpret the
change of isomorphisms into the gluing data, this is exactly the
form given in \cite{LSZ}. Therefore, there is a natural isomorphism
between $(M,\nabla)_{(E,\theta)}$ and
$(M_{\exp},\nabla_{exp})_{(E,-\theta)}$.
\end{proof}

{\itshape Step 2.} Let $M\in \mathcal{MF}^{\nabla}_{[0,n]}(X), n\leq
p-2$ with $pM=0$. We remark that Faltings category exists also for
$n\leq p-1$ (see Theorem 2.3 \cite{Fa2}). But in several places of
\cite{SXZ} invoking the Taylor formula, the assumption $n\leq p-2$
has been explicitly used. So we have to keep the assumption $n\leq
p-2$ here. Let $(G,\theta)$ be a Higgs subbundle of
$(E,\theta)=Gr_{Fil^{\cdot}}(M,\nabla)$ which is nilpotent of
exponent $\leq p-2$. In \cite{SXZ}, we associate $(G,\theta)$ a de
Rham subbundle $(M_{(G,\theta)},\nabla)$ of $(M,\nabla)$ by a local
lifting and gluing process.
\begin{claim}
There is an isomorphism $\phi:
(M_{exp},\nabla_{exp})_{(E,\theta)}\cong (M,\nabla)$ such that for
any Higgs subbundles $G\subset E$,
$$
\phi[(M_{exp},\nabla_{exp})_{(G,\theta)}]=(M_{(G,\theta)},\nabla).
$$
\end{claim}
\begin{proof}
This follows from Proposition 5 \cite{LSZ} and its proof.
\end{proof}
We summarize the previous discussions into the following statement. Recall the notations in \S\ref{section on periodic Higgs subbundles in char
p}: $(\sX,\sS)=(X_0/k,X_1'/W_2(k))$ and $\pi: X_0'\to X_0$ is the natural map setting in the Cartesian diagram of the base change.
\begin{proposition}
Let $C_{\sX/\sS}^{-1}$ be the inverse Cartier transform of Ogus and
Vologodsky \cite{OV}. For an $M\in \mathcal{MF}^\nabla$ with $pM=0$
let $(E,\theta)=Gr_{Fil^\cdot}(M,\nabla)$ the associated Higgs
bundle over $X_0$. Then there is an isomorphism of flat bundles
$$
\psi: C_{\sX/\sS}^{-1}\pi^*(E,-\theta)\cong (M,\nabla)
$$
such that for any Higgs subbundle $(G,\theta)\subset (E,\theta)$,
$$
\psi[C_{\sX/\sS}^{-1}\pi^*(G,-\theta)]=(M_{(G,\theta)},\nabla),
$$
where $(M_{(G,\theta)},\nabla)$ is the associated de Rham subbundle
of $(M,\nabla)$ constructed in \cite{SXZ}.
\end{proposition}


\begin{thebibliography}{X-X00}


\bibitem{Car} H. Carayol, Sur la mauvaise r\'{e}duction des courbes de
Shimura, Comp. Math. 59 (1986), no. 2, 151-230.

\bibitem{DW}
C. Deninger, A. Werner, Vector bundles on $p$-adic curves and
parallel transport, Ann. Scient. \'{E}c. Norm. Sup. 38 (2005),
553-597.

\bibitem{De}
P. Deligne, Un th\'{e}or\`{e}me de finitude pour la monodromie,
Discrete Groups in Geometry and Analysis, Birkhauser 1987, 1-19.

\bibitem{Fa2}
G. Faltings, Crystalline cohomology and $p$-adic
Galois-representations, Algebraic analysis, geometry, and number
theory (Baltimore, MD, 1988), 25-80, Johns Hopkins Univ. Press,
Baltimore, MD, 1989.


\bibitem{Fa1}
G. Faltings, Integral crystalline cohomology over very ramified
valuation rings, Journal of the AMS, Vol. 12, no. 1, 117-144, 1999.

\bibitem{Fa3}
G. Faltings, A $p$-adic Simpson correspondence, Advances in
Mathematics 198 (2005), 847-862.

\bibitem{FL}
J.-M. Fontaine, G. Laffaille, Construction de repr\'{e}sentation
$p$-adiques, Ann. Sci. Ec. Norm. Sup. 15 (1982), 547-608.


\bibitem{Gr}
P. Griffiths, Periods of integrals on algebraic manifolds III, Publ.
Math. I.H.E.S., 38 (1970) 125-180.


\bibitem{LSZ}
G.-T. Lan, M. Sheng, K. Zuo, An inverse Cartier transform via
exponential in positive characteristic, arXiv: 1205.6599, 2012.

\bibitem{LSZ1}
G.-T. Lan, M. Sheng, K. Zuo, Semistable Higgs bundles and
representations of algebraic fundamental groups: Positive
characteristic case, arXiv: 1210.8280, 2012.

\bibitem{LS}
H. Lange, U. Stuhler, Vektorb\"{u}ndel auf Kurven und Darstellungen
der algebraischen Fundamentalgruppe, Math. Z. 156 (1977), 73-83.

\bibitem{O}
A. Ogus, $F$-crystals, Griffiths transversality, and the Hodge
decomposition, Ast\'{e}risque No. 221 (1994).

\bibitem{O1}
A. Ogus, Private communication dated on August 12th, 2012.

\bibitem{OV}
A. Ogus, V. Vologodsky, Nonabelian Hodge theory in characteristic
$p$, Publ. Math. Inst. Hautes \'{e}tudes Sci. 106 (2007), 1-138.



\bibitem{Si}
C. Simpson, Higgs bundles and local systems, Publ. Math. Inst.
Hautes \'{e}tud. Sci. 75 (1992), 5-95.


\bibitem{SZZ}
M. Sheng, J.-J. Zhang, K. Zuo, Higgs bundles over the good reduction
of a quaternionic Shimura curve,  J. reine angew. Math., DOI
10.1515, 2011.


\bibitem{SXZ}
M. Sheng, H. Xin, K. Zuo, A note on the characteristic $p$
nonabelian Hodge theory in the geometric case, arXiv: 1202.3942,
2012.

\bibitem{X}
H. Xin, On Fontaine modules and $F$-$T$ crystals, Doctor Thesis, University of Mainz, in preparation.

\end{thebibliography}
\end{document}